\documentclass[final,leqno]{siamltex}

\usepackage{amsfonts,amssymb} 
\usepackage{amsmath} 
\usepackage{color} 
\usepackage{graphicx} 
\usepackage{epsfig,mathrsfs} 
\usepackage[bf,SL,BF]{subfigure} 
\usepackage{fancyhdr} 
\usepackage{CJK} 
\usepackage{caption} 
\usepackage{wrapfig} 
\usepackage{cases} 
\usepackage{bm}
\usepackage{algorithm}
\usepackage{algorithmic}
\newtheorem{remark}{Remark}[section] 
\newtheorem{example}{Example}[section] 

\title{A Riesz basis Galerkin method for the tempered fractional Laplacian
\thanks{This work was supported by the National Natural Science
Foundation of China under Grant No. 11671182,
and the Fundamental Research Funds for the Central
Universities under Grant No. lzujbky-2017-ot10 and lzujbky-2017-ct01.
The last author (GEK) would like to acknowledge support by
the OSD/ARO/MURI on ¡ÈFractional PDEs for Conservation
Laws and Beyond: Theory, Numerics and Applications (W911NF-15-1-0562).¡É
}}

\author{Zhijiang Zhang\footnotemark[2]
\and Weihua Deng\footnotemark[2]\thanks{ School of Mathematics and Statistics, Gansu Key Laboratory of Applied Mathematics and Complex Systems, Lanzhou University, Lanzhou 730000, P.R. China (Email: dengwh@lzu.edu.cn).}
 \and  George Em Karniadakis\footnotemark[3]\thanks{Division of Applied Mathematics, Brown University, 182 George, Providence, RI 02912, USA (E-mail: george\underline{\hspace{0.5em}}karniadakis@brown.edu).} }
\begin{document}

\maketitle

\begin{abstract}
The fractional Laplacian $\Delta^{\beta/2}$ is the generator of $\beta$-stable L\'evy process, which is the scaling
 limit of the L\'evy fight. Due to the divergence of the second moment of the jump length of the L\'evy fight it is not
 appropriate as
a physical model in many practical applications. However, using a parameter $\lambda$ to exponentially temper the isotropic power law measure of the jump length
leads to the tempered L\'evy fight, which has finite second moment. For  short time the tempered L\'evy fight exhibits the
dynamics of L\'evy fight while after sufficiently long time it turns to normal diffusion. The generator of tempered
 $\beta$-stable L\'evy process is the tempered fractional Laplacian $(\Delta+\lambda)^{\beta/2}$
[W.H. Deng, B.Y. Li, W.Y. Tian, and P.W. Zhang, Multiscale Model. Simul., in press, 2017].
In the current work, we present new computational methods for the tempered fractional Laplacian equation,
 including the cases with the homogeneous and nonhomogeneous generalized Dirichlet type boundary conditions.
We prove the well-posedness of the Galerkin weak formulation and provide convergence analysis of
the single scaling B-spline and multiscale Riesz bases finite element methods.
We propose a technique for efficiently  generating the entries of the dense stiffness matrix
and for solving the resulting algebraic equation by preconditioning. We also present several
numerical experiments to verify the theoretical results.
\end{abstract}


\pagestyle{myheadings}
\thispagestyle{plain}
\markboth{Z. J. Zhang, W. H. Deng, and  G. E. Karniadaki}{VARIATIONAL METHOD FOR TEMPERED FRACTIONAL LAPLACIAN}

\section{Introduction}\label{sec1}
Phenomena of anomalous diffusion are ubiquitous in nature \cite{Metzler:00}. L\'evy flights with isotropic power law measure $|x|^{-n-\beta}$ of the jump length display superdiffusion, where $n$ is the dimension of space and $\beta \in (0,2)$ is a parameter. The scaling limit of L\'evy flight is the $\beta$-stable L\'evy process, the generator of which is the fractional Laplacian $\Delta^{\beta/2}$. This topic has recently become popular in both  pure and applied mathematical communities \cite{Pozrikidis:16}. The divergence of second moment of the L\'evy flight is associated with the possible infinite speed of the motion of the particles, which contradicts their nonzero masses, i.e., the pure power law distribution of jump length sometimes makes the L\'evy flight not a suitable physical model. Hence, tempering the distribution of the jump length becomes a natural idea, namely, modify $|x|^{-n-\beta}$ as  $e^{-\lambda |x|}|x|^{-n-\beta}$ with $\lambda$ being a small nonnegative real number, so that we can obtain the tempered L\'evy flight.

For small $\lambda$, the tempered L\'evy flight exhibits a slow transition of the dynamics from L\'evy flight to
 normal diffusion, which may occur after sufficient long time.
 The scaling limit of the tempered L\'evy flight is called tempered L\'evy process, the generator of which is the tempered fractional
Laplacian $(\Delta+\lambda)^{\beta/2}$ \cite{Deng:17}. In this paper, we mainly focus on developing numerical methods in the Riesz basis Galerkin framework
for the tempered fractional Laplacian, i.e.
\begin{eqnarray}\label{Laplacedirichelt}
\left\{\begin{array}{ll}
-(\Delta +\lambda)^ {\beta/2}p(x)=f(x),&~~ x\in \Omega,\\
p(x)=0,&~~x\in \mathbb{R}\backslash\Omega,
\end{array}\right.
\end{eqnarray}
which corresponds to the one-dimensional case of the initial and boundary value problem in Eq. (49) recently proposed in \cite{Deng:17}.
Here  $\beta\in (0,2),\,\lambda\ge0, \,\Omega=(a,b), \, f(x)\in H^{-\beta/2}(\Omega)$,  and
\begin{eqnarray}\label{Laplacianequation}
(\Delta +\lambda)^{\beta/2}p(x):=-c_{\beta} {~\rm P.V.~} \int_{\mathbb{R}}\frac{p(x)-p(y)}{e^{\lambda|x-y|}|x-y|^{1+\beta}}dy
\end{eqnarray}
with
\begin{eqnarray}
c_{\beta}=\left\{\begin{array}{ll}
\frac{\beta\Gamma(\frac{1+\beta}{2})}{2^{1-\beta}\pi^{1/2}\Gamma(1-\beta/2)}& ~{\rm for~}\lambda=0
 ~{\rm or}~ \beta=1,\\[3pt] \frac{\Gamma(\frac{1}{2})}{2\pi^{\frac{1}{2}}|\Gamma(-\beta)|} &~{\rm for}~\lambda>0~{\rm and}~\beta\not=1,
\end{array}\right.
\end{eqnarray}
where P.V. denotes the Cauchy principle value, being the limit of the integral over $\mathbb{R}\backslash B_{\epsilon} (x)$ as $\epsilon\to 0$;
the definition of this form is indeed necessary when $\beta\ge 1$.

Obviously, when $\lambda=0$, (\ref{Laplacianequation}) reduces to fractional Laplacian
\begin{eqnarray} \label{eqnarrddddd}
\Delta^{\beta/2}p(x):=-c_{\beta} {~\rm P.V.~} \int_{\mathbb{R}}\frac{p(x)-p(y)}{|x-y|^{1+\beta}}dy,
\end{eqnarray}
which has the Fourier transform (assuming that $\mathscr{F}[\Delta^{\beta/2} p(x)](\xi)$ and $\mathscr{F}[p(x)](\xi)$ exist)
 \begin{eqnarray}
 \mathscr{F}[\Delta^{\beta/2} p(x)](\xi)=-|\xi|^\beta\mathscr{F}[p(x)](\xi).
 \end{eqnarray}
Here  $\mathscr{F}[w(x)](\xi),\,\xi\in \mathbb{R}$ is defined by $\mathscr{F}[w(x)](\xi)=\int_{\mathbb{R}}w(x)e^{-ix\xi}dx$,
and for $w_1,w_2\in L^2(\mathbb{R})$, the Parseval identity \cite[pp. 100]{Guo:15} can be applied
\begin{eqnarray}\label{Parseval}
\int_{\mathbb{R}}w_1(x) \overline{w_2(x)}dx=\frac{1}{2\pi}\int_{\mathbb{R}}\mathscr{F}[w_1(x)](\xi)\overline{\mathscr{F}[w_2(x)](\xi)}d\xi.
\end{eqnarray}

Recently, the fractional Laplacian has attracted  a lot of attention,
but even in the simplified context \cite{Acosta:17,Acosta:171,DElia:13,Huang:14}
it is far from the well-developed status of the classical Laplacian. The numerical resolution of the  fractional Laplacian
involves two major challenging tasks, namely the singular kernel and the integration in an unbounded region.
For the finite difference method the convergence rate is even influenced by the regularity of the exact solution outside of
 the domain $\Omega$ \cite{Huang:14}. As for the tempered fractional differential equations,
 there are some published works on numerical methods \cite{Baeumer:10,Hanert:14,Lican:15,Zayernourt:15},
 but no  theoretical results under the variational framework exists. In the current paper we prove the well-posedness of the variational
formulation of  (\ref{Laplacedirichelt}), where extra efforts must be made to obtain the $\tilde{H}^{\beta/2}(\mathbb{R})$-coercivity.
Subsequently, the convergence analysis and the effective implementation of the finite dimensional approximation with the single-scale
or multiscale basis functions are presented, in which the properties of Riesz basis and multiresolution are used.

 The rest of this paper is organized  as follows. In Section \ref{sec2},
we introduce the function spaces and the properties of the tempered fractional Laplacian to be used.
The variational formulation of (\ref{Laplacedirichelt}) and its well-posedness are presented and discussed in Section \ref{sec3}.
We develop the Riesz basis Galerkin approximation and perform its convergence analysis in Section \ref{sec4}.
Section \ref{sec5} provides the effective implementations, including calculating the entries of the stiffness matrix
and solving the resulting algebraic equations. We discuss the model (\ref{Laplacedirichelt})
 with nonhomogeneous generalized Dirichlet type boundary condition in Section \ref{sec6}.
 The numerical results are given in Section \ref{sec7} and we conclude the paper with remarks in Section \ref{sec8}.


\section{Preliminaries}\label{sec2}
Throughout the paper by the notation $A\lesssim B$ we mean that $A$ can be bounded by a multiple of $B$,
independent of the parameters they may depend on, while the expression $A\simeq B$ means that $A\lesssim B\lesssim A$.
Let $E$ be an open set of $\mathbb{R}$. If $s\ge 0$ is a nonnegative integer, we denote by $H^s(E)$ the classical Sobolev
space equipped with the norm
\begin{eqnarray}
\|w\|_{H^s(E)}:=\left(\sum_{0\le k\le s}\left\|w^{(k)}\right\|^2_{L^2(E)}\right)^{1/2},
\end{eqnarray}
where $w^{(k)}$ stands for the $k$-th distributional derivative, and $H^0(E):=L^2(E)$. In the following, we define the fractional Sobolev spaces, where $s$ is not an integer.

For a fixed $s\in (0,1)$, the Sobolev space $H^s(E)$ is defined as
\begin{eqnarray}
H^s(E):=\left\{w\in L^2(E): |w|_{H^s(E)}<\infty    \right\},
\end{eqnarray}
where
\begin{eqnarray}
|w|_{H^s(E)}:=\left(\int_{E}\int_{E}\frac{\left|w(x)-w(y)\right|^2}{|x-y|^{1+2s}}dx dy\right)^{1/2}
\end{eqnarray}
is the Slobodeckii semi-norm \cite[pp. 74]{McLean:10} of $w(x)$ . The space $H^s(E)$ is a Banach space, endowed with the natural norm
\begin{eqnarray}
\left\|w\right\|_{H^s(E)}:=\left(\left\|w\right\|_{L^2(E)}^2+|w|^2_{H^s(E)}\right)^{1/2}.
\end{eqnarray}
Indeed, $H^s(E)$  also is a  Hilbert space \cite[pp. 75]{McLean:10}.
 For $s>1$ and $s\notin \mathbb{N}$, we can define  $H^s(E)$ as follows:
\begin{eqnarray}
H^s(E):=\left\{w\in H^{\lfloor s \rfloor}(E): \left|w^{(\lfloor s \rfloor)}\right|_{H^{s-\lfloor s \rfloor}(E)}<\infty\right\},
\end{eqnarray}
where $\lfloor s \rfloor$ is the biggest integer smaller than $s$. In this case, $H^s(E)$ is endowed with the norm
\begin{eqnarray}
\left\|w\right\|_{H^s(E)}:=\left(\left\|w\right\|^2_{H^{\lfloor s \rfloor}(E)}+\left|w^{(\lfloor s \rfloor)}\right|^2_{H^{s-\lfloor s \rfloor}(E)}\right)^{1/2}.
\end{eqnarray}
We note that $H^s(E)$ is a well-defined Banach space for every $s\ge 0$. Moreover, when $E=\mathbb{R}$, for $0<s<1$, it holds that $\left|w\right|^2_{H^s(\mathbb{R})}\simeq \int_{\mathbb{R}}|\xi|^{2s}|\mathscr{F}[w](\xi)|^2d\xi$,
and for $s>0$,  $\left\|w\right\|^2_{H^s(\mathbb{R})}\simeq \int_{\mathbb{R}}\left(1+|\xi|^{2s}\right)|\mathscr{F}[w](\xi)|^2d\xi $ \cite[pp. 79-80]{McLean:10}.
In fact, the Sobolev space $H^{s}(\mathbb{R})$ can also be defined as
\begin{eqnarray}
H^{s}(\mathbb{R})=\left\{w\in L^2(\mathbb{R}): \int_{\mathbb{R}}\left(1+|\xi|^{2s}\right)|\hat{w}(\xi)|^2d\xi  <\infty \right\}.
\end{eqnarray}
Let  $C_0^\infty(E)$  be the space of functions that are infinite differentiable on $E$ and have compact support in $E$.
Then $C_0^{\infty}(\mathbb{R})$ is dense in $H^s(\mathbb{R})$. However, if $E\subset\mathbb{R}$ is strict, the space $C_0^{\infty}(E)$ generally
is not dense in $H^s(E)$. Hence, we denote by $H_0^s(E)$ the closure of $C_0^{\infty}(E)$ in $H^s(E)$. As usual, $H^{-s}(E)$ is the dual space of $H_0^s(E)$.
 In addition, let $\Omega=(a,b)$ be a nonempty open interval of $\mathbb{R}$.
By $\tilde{C}_0^{\infty}(\Omega)$ we denote the space of all infinitely differentiable functions on $\mathbb{R}$ whose support is compact
and contained in $\Omega$. For $s>0$, we use $\tilde{H}_0^s(\Omega)$ to denote the closure of $\widetilde{C}_0^{\infty}(\Omega)$ in $H^s(\mathbb{R})$.
 For $s=0$, $\widetilde{H}_0^{s}(\Omega)$ is interpreted as the closure of $\widetilde{C}_0^{\infty}(\Omega)$ in $L^2(\mathbb{R})$,
and denoted as $\widetilde{L}^2(\Omega)$. Obviously, $\widetilde{H}_0^s(\Omega)\subset H_0^s(\Omega)$. Moreover,
by \cite{Fiscella:15}, when $s\in (0,1)$, $\widetilde{H}^s_0(\Omega)$ can also be defined by
\begin{eqnarray}
\widetilde{H}_0^{s}(\Omega)=\left\{w(x)\in H^s(\mathbb{R}):  w(x)=0 ~{\rm a.e.~for~} x\in \mathbb{R}\backslash  \Omega\right\}.
\end{eqnarray}

Here, the space $\widetilde{C}_0^{\infty}(\Omega)$  is  actually the space $C_c^{\infty}(a,b)$ in  \cite[pp. 178]{Jia:09} and
 the space $C_0^{\infty}(a,b)$ in \cite[pp. 237]{Fiscella:15}.
The space $\widetilde{H}_0^{s}(\Omega)$ is the space $H_0^\mu(a,b)$ with $\mu=s$ in \cite[pp. 178]{Jia:09}
 and the space $X_0^{s,p}(a,b)$ with $p=2$  in \cite[pp. 236--237]{Fiscella:15}.

Next, we give some properties of the tempered fractional Laplacian.
\begin{proposition}\label{sec1lemma1}
For $w(x)\in C_0^\infty(\mathbb{R})$ and $\lambda >0 $,  we have
  \begin{eqnarray}\label{lemma1eq1}
 &&\mathscr{F}\left[(\Delta +\lambda\right)^{\beta/2}w(x)](\xi)\nonumber\\
 &&~=(-1)^{\lfloor \beta \rfloor}\left(\lambda^\beta-\left(\lambda^2+\left|\xi\right|^2\right)^{\frac{\beta}{2}}\cos\left(\beta \arctan\left(\frac{\left|\xi\right|}{\lambda}\right)\right)\right)\mathscr{F}[w](\xi)
\end{eqnarray}
for $\beta\in (0,1)\cup(1,2)$, and
\begin{eqnarray}
&&\mathscr{F}\left[(\Delta +\lambda\right)^{\beta/2}w(x)](\xi)\nonumber\\
&&~~=\frac{2}{\pi}\left(-\left|\xi\right|\arctan\left(\frac{\left|\xi\right|}{\lambda}\right)
+\frac{\lambda}{2}\ln(\lambda^2+\left|\xi\right|^2)-\lambda\ln(\lambda)\right)\mathscr{F}[w](\xi)
\end{eqnarray}
for $\beta=1$, where $\lfloor \beta \rfloor:=\left\{z\in \mathbb{N}: 0\le\beta-z<1\right\}$.
\end{proposition}
\begin{proof}
For $w\in C_0^{\infty}(\mathbb{R})$, it holds that
 \begin{eqnarray}\label{eeeeenhjjksd1}
 \left(\Delta+\lambda\right)^{\beta/2} w(x)=\frac{-c_\beta}{2}\int_{\mathbb{R}}\frac{2w(x)-w(x-y)-w(x+y)}{e^{\lambda |y|}|y|^{1+\beta}}dy.
 \end{eqnarray}
 So $\left(\Delta+\lambda\right)^{\beta/2} w(x)$ make sense for $x\in \mathbb{R}$. Then
\begin{eqnarray}
 &&\mathscr{F}\left[(\Delta +\lambda\right)^{\beta/2}w(x)](\xi)=-c_{\beta}\int_{0}^{\infty}\frac{2-e^{-iy\xi}-e^{iy\xi}}{e^{\lambda y}y^{1+\beta}}dy\,\mathscr{F}[w](\xi)\nonumber\\
 &&~~=-c_{\beta}\int_{0}^{\infty}\left(2-2\cos(y\xi)\right)e^{-\lambda y}y^{-1-\beta}dy\,\mathscr{F}[w](\xi)\nonumber\\
 &&~~=\frac{c_{\beta}}{\beta}\int_{0}^{\infty}\left(2-2\cos(y\xi)\right)e^{-\lambda y}d\left(y^{-\beta}\right)\,\mathscr{F}[w](\xi)\nonumber\\
 &&~~=\frac{2c_{\beta}}{-\beta}\int_0^{\infty}y^{-\beta}e^{-\lambda y}\left(\xi\sin(y\xi)-\lambda(1-\cos(y\xi))\right)dy\,\mathscr{F}[w](\xi).
\end{eqnarray}
Since $\xi\sin(y\xi)-\lambda(1-\cos(y\xi)$  is an even function w.r.t. $\xi$, in the following we assume  $\xi\ge 0$.

If $0<\beta<1$, we have
\begin{eqnarray}
&&\int_0^{\infty}y^{-\beta}e^{-\lambda y}\left(\xi\sin(y\xi)-\lambda(1-\cos(y\xi))\right)dy\nonumber\\
&&~=\xi\int_0^{\infty}y^{-\beta}e^{-\lambda y}\sin(y\xi)dy-\lambda\int_0^{\infty}y^{-\beta}e^{-\lambda y}dy+\lambda \int_0^{\infty}y^{-\beta}e^{-\lambda y}\cos(y\xi)dy\nonumber\\
&&~=\frac{\Gamma(1-\beta)\xi}{\left(\lambda^2+\xi^2\right)^{\frac{1-\beta}{2}}}
\sin\left((1-\beta)\arctan\left(\frac{\xi}{\lambda}\right)\right)-\lambda^\beta\int_0^{\infty}y^{-\beta}e^{-y}dy\nonumber\\
&&~~~+\frac{\Gamma(1-\beta)\lambda}{\left(\lambda^2+\xi^2\right)^{\frac{1-\beta}{2}}}\cos\left((1-\beta)\arctan\left(\frac{\xi}{\lambda}\right)\right)\nonumber\\
&&~={\Gamma(1-\beta)}{\left(\lambda^2+\xi^2\right)^{\frac{\beta}{2}}}\cos\left(\beta \arctan\left(\frac{\left|\xi\right|}{\lambda}\right)\right)
-\lambda^\beta \Gamma(1-\beta),
\end{eqnarray}
where the formulae  \cite[Eq. (3.944(5))]{Gradshteyn:80} and \cite[Eq. (3.944(6))]{Gradshteyn:80} have been used in the second step.

For $1<\beta<2$, using the integration by parts again, and similarly we have
\begin{eqnarray}
&&\int_0^{\infty}y^{-\beta}e^{-\lambda y}\left(\xi\sin(y\xi)-\lambda(1-\cos(y\xi))\right)dy\nonumber\\
&&~=\frac{1}{1-\beta}\int_0^{\infty}e^{-\lambda y}y^{1-\beta}\left(\left(\lambda^2-\xi^2\right)\cos(y\xi)+2\lambda
\xi \sin(y\xi)-\lambda^2\right)dy\nonumber\\
&&~=\frac{\Gamma(1-\beta)}{(\lambda^2+\xi)^{1-\frac{\beta}{2}}}
\bigg(\left(\lambda^2-\xi^2\right)\cos\left(\left(2-\beta\right)\arctan\left(\frac{\xi}{\lambda}\right)\right)\nonumber\\
&&~~~+2\lambda\xi\sin\left(\left(2-\beta\right)\arctan\left(\frac{\xi}{\lambda}\right)\right)\bigg)-\Gamma(1-\beta)\lambda^\beta\nonumber\\
&&~={\Gamma(1-\beta)}{\left(\lambda^2+\xi^2\right)^{\frac{\beta}{2}}}\cos\left(\beta \arctan
\left(\frac{\left|\xi\right|}{\lambda}\right)\right)-\lambda^\beta \Gamma(1-\beta).
\end{eqnarray}

For $\beta=1$, using the integration by parts, we have
\begin{eqnarray}
&&\int_0^{\infty}y^{-\beta}e^{-\lambda y}\left(\xi\sin(y\xi)-\lambda(1-\cos(y\xi))\right)dy\nonumber\\
&&~=\left(\lambda^2-\xi^2\right)\int_0^{\infty}\ln(y)e^{-\lambda y}\left(\cos(y\xi)\right)dy\nonumber\\
&&~~~+2\lambda\xi\int_0^{\infty}\ln(y)e^{-\lambda y}\sin(u\xi)dy-\lambda^2\int_0^{\infty}\ln(y)e^{-\lambda y}dy\nonumber\\
&&~=\frac{2\lambda\xi}{\lambda^2+\xi^2}\left(\lambda \arctan\left(\frac{\xi}{\lambda}\right)-\gamma\xi
-\frac{\xi}{2}\ln(\lambda^2+\xi^2)\right)\nonumber\\
&&~~~+\frac{\xi^2-\lambda^2}{\lambda^2+\xi^2}\left(\frac{\lambda}{2}\ln\left(\lambda^2+\xi^2\right)
+\xi\arctan\left(\frac{\xi}{\lambda}\right)+\lambda \gamma\right)+\lambda\left(\gamma+\ln(\lambda)\right)\nonumber\\
&&~=\xi\arctan\left(\frac{\xi}{\lambda}\right)-\frac{\lambda}{2}\ln(\lambda^2+\xi^2)+\lambda\ln(\lambda),
\end{eqnarray}
where $ \gamma$ denotes the Euler constant, and the formulae   \cite[Eq. (4.441(1))-(4.441(2))]{Gradshteyn:80}
and \cite[Eq. (4.331(1))]{Gradshteyn:80} have been used in the second step.
\end{proof}

The  following proposition  is  similar to Theorem 2.1 of  \cite{Jin:15}.
\begin{proposition}\label{section1propostion11}
For $0<\beta<2$, the tempered fractional Laplacian  can be extended to
 a continuous linear mapping from $H^\beta(\mathbb{R})$ to $L^2(\mathbb{R})$, and its Fourier transform remains the same.
 \end{proposition}
 \begin{proof}
 Firstly, assume  $w\in C_0^{\infty}(\mathbb{R})$.
For $\lambda=0$, by
\begin{eqnarray}
 \mathscr{F}[\Delta^{\beta/2} w(x)](\xi)=-\left|\xi\right|^{\beta}\mathscr{F}\left[w(x)\right](\xi),
 \end{eqnarray}
 and the Parseval identity (\ref{Parseval}), we have
 \begin{eqnarray}
 \left\|\Delta^{\beta/2} w\right\|^2_{L^2(\mathbb{R})}=\frac{1}{2\pi}\int_{\mathbb{R}}\left|\xi\right|^{2\beta}
 \left|\mathscr{F}\left[w(x)\right](\xi)\right|^2 d\xi\le\left\|w\right\|^2_{H^{\beta}(\mathbb{R})}.\label{Laplaceextension}
 \end{eqnarray}
 For $\lambda>0$, by Proposition \ref{sec1lemma1} and the Parseval identity (\ref{Parseval}), we have
 \begin{eqnarray}
  &&\left\|\left(\Delta+\lambda\right)^{\beta/2} w\right\|^2_{L^2(\mathbb{R})}
  \le \int_{\mathbb{R}}\mathcal{G}^2(\lambda,\xi,\beta)\left|\mathscr{F}[w(x)](\xi)\right|^2d\xi\nonumber\\
  &&~\lesssim \int_{\mathbb{R}}\left(1+|\xi|^{2\beta}\right)\left|\mathscr{F}[w(x)](\xi)\right|^2d\xi
  \lesssim \left\|w\right\|^2_{H^{\beta}(\mathbb{R})},\label{ddmmmddee994322}
 \end{eqnarray}
 where
 \begin{eqnarray}\label{intergraeliqneqq}
 ~~~~\mathcal{G}(\lambda,\xi,\beta):=
 \left\{\begin{array}{l}(-1)^{\lfloor \beta \rfloor}\left(\left(\lambda^2+\left|\xi\right|^2\right)^{\frac{\beta}{2}}
 \cos\left(\beta \arctan(\frac{\left|\xi\right|}{\lambda})\right)-\lambda^\beta\right),~\beta\not=1,\\
 \frac{2}{\pi}\left(\left|\xi\right|\arctan\left(\frac{\left|\xi\right|}{\lambda}\right)
-\frac{\lambda}{2}\ln(\lambda^2+\left|\xi\right|^2)+\lambda\ln(\lambda)\right),~~~\beta=1;
\end{array}\right.
 \end{eqnarray}
 and the inequalities $\left(x_1+x_2+x_3\right)^2\le 3\left(x_1^2+x_2^2+x_3^2\right),
  \,\cos\left(\beta \arctan\left(\frac{\left|\xi\right|}{\lambda}\right)\right)\le 1$,
   $\arctan\left(\frac{\left|\xi\right|}{\lambda}\right)\le \frac{\pi}{2},\,
\ln(\lambda^2+|\xi|^2)\le \lambda^2+|\xi|^2$, and $\left(\lambda^2+|\xi|^2\right)^{\beta}\le 2^{\beta}\left(\lambda^{2\beta}+|\xi|^{2\beta}\right)$,
have been used. Then by the density of  $C_0^{\infty}(\mathbb{R})$  in $H^\beta(\mathbb{R})$,
one can continuously extend $\left(\Delta+\lambda\right)^{\beta/2}$ to an operator from  $H^\beta(\mathbb{R})$  to $L^2(\mathbb{R})$.

Secondly, let  $w\in H^{\beta}(\mathbb{R})$. Then $\mathscr{F}[\left(\Delta+\lambda\right)^{\beta/2}w(x)](\xi) \in L^2(\mathbb{R})$,
 and there exists a sequence $\left\{w_k\right\}\in C_0^\infty(\mathbb{R})$
 such  that $\lim_{k\to\infty}\left\|w-w_k\right\|_{H^\beta(\mathbb{R})}=0$. Therefore, for $\lambda>0$, by (\ref{ddmmmddee994322}) and the Parseval identity  we have
 \begin{eqnarray}
 &&\left\|-\mathcal{G}(\lambda,\xi,\beta)\mathscr{F}[w(x)](\xi)-\mathscr{F}[\left(\Delta+\lambda\right)^{\beta/2}w(x)](\xi)\right\|_{L^2(\mathbb{R})}
 \nonumber\\
 &&~\le \left\| -\mathcal{G}(\lambda,\xi,\beta)\mathscr{F}\left[w(x)-w_k\right](\xi)\right\|_{L^2(\mathbb{R})}\nonumber\\
 &&~~~+\left\| \mathscr{F}\left[\left(\Delta+\lambda\right)^{\beta/2}(w_k-w)(x)\right](\xi)\right\|_{L^2(\mathbb{R})}\nonumber\\
 &&~\le 2\left\|w-w_k\right\|_{H^\beta(\mathbb{R})}\to 0.
 \end{eqnarray}
 Thus  $\mathscr{F}\left[\left(\Delta+\lambda\right)^{\beta/2}w(x)\right](\xi)=-\mathcal{G}(\lambda,\xi,\beta)\mathscr{F}[w(x)](\xi)$.
 The proof for $\lambda=0$ is similar.
\end{proof}

%
%
%
%
%
%
%

\begin{proposition}
If $w(x)\in C^3(\mathbb{R})$ and $w(x),\,w^{(3)}(x)\in L^\infty(\mathbb{R})$, for $\lambda\ge 0$, we have
 \begin{eqnarray}\label{lemma1eq2}
  \lim_{\beta\to 2^-}(\Delta +\lambda)^ {\beta/2}w(x)=\frac{d^2 w(x)}{dx^2}.
\end{eqnarray}
\end{proposition}

\begin{proof}
For $\lambda>0$, following the definition of tempered fractional Laplacian, we have
\begin{eqnarray}
&&\left|{c_\beta}\int_0^\infty \frac{w(x+y)-2w(x)+w(x-y)-w^{(2)}(x)y^2}{e^{\lambda y}y^{1+\beta}}dy\right|\nonumber\\
&&~\le c_{\beta}\big\|w^{(3)}\big\|_{L^\infty(\mathbb{R})}\int_0^\infty\frac{y^3}{y^{1+\beta}e^{\lambda y}}dy\nonumber\\
&&~=c_{\beta}\lambda^{\beta-3}\Gamma(3-\beta)\big\|w^{(3)}\big\|_{L^\infty(\mathbb{R})}\to 0, ~~~\beta\to 2^{-},\nonumber
\end{eqnarray}
where $c_\beta=\frac{\Gamma(\frac{1}{2})}{2\pi^{\frac{1}{2}}|\Gamma(-\beta)|}$. Further note that \begin{eqnarray*}
	&&{c_\beta}\int_0^\infty \frac{w^{(2)}(x)y^2}{e^{\lambda y}y^{1+\beta}}dy=w^{(2)}(x)c_{\beta}\lambda^{\beta-2}\Gamma(2-\beta)\nonumber\\
	&&~=w^{(2)}(x)\frac{(1-\beta)(-\beta)}{2}\to w^{(2)}(x), ~~\beta\to 2^{-},
\end{eqnarray*}
which results in the desired result.

For the case $\lambda=0$, note that
\begin{eqnarray*}
&&\left|{c_\beta}\int_c^\infty \frac{w(x+y)-2w(x)+w(x-y)}{y^{1+\beta}}dy\right|\nonumber\\
&&~=4c_{\beta}\left\|w\right\|_{L^{\infty}(\mathbb{R})}\int_c^\infty \frac{1}{y^{1+2s}}dy\to 0, ~~~\beta\to 2^{-};\nonumber\\[5pt]
&&\left|{c_\beta}\int_0^c \frac{w(x+y)-2w(x)+w(x-y)-w^{(2)}(x)y^2}{e^{\lambda y}y^{1+\beta}}dy\right|\nonumber\\
&&~=c_{\beta}\frac{c^{3-\beta}}{3-\beta}\big\|w^{(3)}\big\|_{L^\infty(\mathbb{R})}\to 0, ~~~\beta\to 2^{-},
\end{eqnarray*}
where $c\ge 0$  is an arbitrary given constant, and $c_\beta=\frac{\beta\Gamma(\frac{1+\beta}{2})}{2^{1-\beta}\pi^{1/2}\Gamma(1-\beta/2)}$. Therefore,
\begin{eqnarray*}
\left(\Delta +\lambda\right)^ {\beta/2} w(x)={c_\beta}\int_0^d \frac{w^{(2)}(x)y^2}{y^{1+\beta}}dy=\frac{c_\beta}{2-\beta}d^{2-\beta}w^{(2)}(x).
\end{eqnarray*}
By \cite[pp. 25]{Pozrikidis:16}
\begin{eqnarray*}
c_\beta=\frac{1}{\pi}\Gamma(1+\beta)\sin\left(\frac{\beta\pi}{2}\right)=\frac{1}{\pi}\Gamma(1+\beta)\sin\left(\frac{(2-\beta)\pi}{2}\right),
\end{eqnarray*}
we have  $\lim_{\beta\to 2^{-}}\frac{c_\beta}{2-\beta}d^{2-\beta}w^{(2)}(x)=w^{(2)}(x)$.
\end{proof}

Equation (\ref{lemma1eq2}) shows that if $\beta\to 2^-$,  the tempered fractional Laplacian coincides with the classical Laplacian.
Finally, we give the concept of Riesz basis that will be used later.
\begin{definition}[\cite{Jia:09,Urban:09}]
A countable collection of elements $\mathcal{E}:=\left\{e_i\right\}_{i\in I}\,(I\subset \mathbb{Z})$ of a Hilbert space $H$ is called a Riesz basis of $H$ if each element in $H$ has an expansion in terms of $\mathcal{E}$ and there exists (Riesz) constants $0<A\le B<\infty$ such that
\begin{eqnarray}
A\sum_{i\in I}|d_i|^2\le \bigg\|\sum_{i\in I}d_ie_i\bigg\|_H^2\le B\sum_{i\in I}|d_i|^2.
\end{eqnarray}
\end{definition}
\section{Weak solution and well-posedness}\label{sec3}
In this section, we first give the definition of the weak solution of (\ref{Laplacedirichelt}),
and then discuss the well-posedness of the  corresponding weak formulation. As in the usual approach in dealing with elliptic PDE,
 multiplying  both sides of (\ref{Laplacedirichelt}) by $v\in \widetilde{C}^\infty_0(\Omega)$ and integrating them over $\Omega$ leads to
\begin{eqnarray}
c_\beta\int_\mathbb{R} \int_\mathbb{R}\frac{\left(p(x)-p(y)\right)v(x)}{e^{\lambda|x-y|}\left|x-y\right|^{1+\beta}}dydx =\int_\Omega f(x) v(x)dx.
\end{eqnarray}
Instead of performing integration by parts, we use the fact
\begin{eqnarray}
\int_\mathbb{R} \int_\mathbb{R}\frac{\left(p(x)-p(y)\right)v(x)}{e^{\lambda|x-y|}\left|x-y\right|^{1+\beta}}dydx =\int_\mathbb{R} \int_\mathbb{R}\frac{\left(p(y)-p(x)\right)v(y)}{e^{\lambda|y-x|}\left|y-x\right|^{1+\beta}}dydx
\end{eqnarray}
to get the weak formulation of (\ref{Laplacedirichelt}): find $p\in  {\widetilde{H}_0^{\beta/2}(\Omega)}$ such that
 \begin{eqnarray}\label{weakformulae}
 B(p,v)=\langle f,v\rangle
 \end{eqnarray}
 for all $v\in {\widetilde{H}_0^{\beta/2}(\Omega)}$, where the duality pairing $\left\langle f,v\right\rangle:=\int_\Omega f(x) v(x)dx$ and the bilinear form
 \begin{eqnarray} \label{bilinearformeddee5555}
 B(p,v):=\frac{c_{\beta}}{2}\int_{\mathbb{R}}\int_{\mathbb{R}}\frac{\left(p(x)-p(y)\right)\left(v(x)-v(y)\right)}
 {e^{\lambda|x-y|}\left|x-y\right|^{1+\beta}}dydx.
 \end{eqnarray}

When $\lambda=0$,  Ref. \cite{Deng:17} gives the weak formulation of (\ref{Laplacedirichelt})  as: find $p\in {\widetilde{H}_0^{\beta/2}(\Omega)}$, such that
 \begin{eqnarray}\label{weakformulae1}
 \int_{\mathbb{R}}\Delta^{\beta/4}p\,\Delta^{\beta/4}v dx=\int_\Omega fv dx ~~~~\forall v\in {\widetilde{H}_0^{\beta/2}(\Omega)},
 \end{eqnarray}
being equivalent to (\ref{weakformulae}) with $\lambda=0$. In fact, it can be simply verified as:
for any $p, v\in \widetilde{H}_0^{\beta/2}(\Omega)$, by the Parseval identity (\ref{Parseval}) and $\mathscr{F} [\Delta^{\beta/4}p(x)](\xi)=-|\xi|^{\beta/2}\mathscr{F}[p(x)] (\xi)$, 
  \begin{eqnarray}
  &&\frac{c_\beta}{2}\int_{\mathbb{R}}\int_{\mathbb{R}}\frac{\left(p(x)-p(y)\right)\left(v(x)-v(y)\right)}{\left|x-y\right|^{1+\beta}}dydx\nonumber\\
  &&~~=\frac{c_\beta}{2}\int_{\mathbb{R}}\int_{\mathbb{R}}\frac{(p(x+y)-p(y))(v(x+y)-v(y))}{|x|^{1+\beta} }dy dx\nonumber\\
  &&~~=\frac{c_\beta}{4\pi}\int_{\mathbb{R}}\int_{\mathbb{R}}\mathscr{F}\left[\frac{p(x+y)-p(y)}{|x|^{1/2+\beta/2} } \right](\xi)
  \,\overline{\mathscr{F}\left[\frac{v(x+y)-v(y)}{|x|^{1/2+\beta/2} } \right](\xi)} d\xi dx\nonumber\\
  &&~~=\frac{c_\beta}{4\pi}\int_{\mathbb{R}}\int_{\mathbb{R}}\frac{|e^{i\xi x}-1|^2}{|x|^{1+\beta}} dx  \,\mathscr{F}[p(y)](\xi)\, \overline{\mathscr{F}[v(y)]
  (\xi)}d\xi\nonumber\\
  &&~~=\frac{1}{2\pi}\int_{\mathbb{R}} \left(-|\xi|^{\beta/2}\mathscr{F}[p(x)](\xi)\right)\left(\overline{-|\xi|^{\beta/2}\mathscr{F}[v(x)](\xi)}\right) d\xi\nonumber\\
  &&~~=\int_{\mathbb{R}}\Delta^{\beta/4}p\,\Delta^{\beta/4}v dx,\label{hkkdkdh34452}
  \end{eqnarray}
  where  the result \cite[pp. 23-28]{Pozrikidis:16}
  \begin{eqnarray}
\int_{\mathbb{R}}\frac{|e^{i\xi x}-1|^2}{|x|^{1+\beta}} dx
=\int_{\mathbb{R}}\frac{2-2\cos(\xi x)}{|x|^{1+\beta}}
=4|\xi|^{\beta}\int_{\mathbb{R}} \frac{\sin^2{\frac{\omega}{2}}}{|\omega|^{1+\beta}}d\omega=\frac{2|\xi|^{\beta}}{c_{\beta}}
  \end{eqnarray}
has been used in the second equality from below. However,  when  $\lambda>0$, (\ref{weakformulae}) does not have the equivalent form like (\ref{weakformulae1}), which can be simply discovered by recalling the proof process of Proposition \ref{sec1lemma1}, i.e.,
  \begin{eqnarray}
  &&B(p,v)
 =\frac{c_\beta}{2\pi}\int_{\mathbb{R}}\int_{0}^{\infty}\frac{2-2\cos(\xi x)}{e^{\lambda|x|}|x|^{1+\beta}} dx  \,\mathscr{F}[p(y)](\xi)\, \overline{\mathscr{F}[v(y)](\xi)}d\xi\nonumber\\
 &&~~~~~~~~~=\frac{1}{2\pi}\int_{\mathbb{R}}\mathcal{G}(\lambda,\xi,\beta)  \,\mathscr{F}[p(y)](\xi)\, \overline{\mathscr{F}[v(y)](\xi)}d\xi.
 \label{hkkmo0dkdh34452}
 \end{eqnarray}
 where $\mathcal{G}(\lambda,\xi,\beta)$ is given in (\ref {intergraeliqneqq}).
On the contrary, for $\beta\in (0,1)\cup(1,2)$, by introducing the operators  ${}_{-\infty}\mathbb{D}_x^{\beta/2,\lambda}$ and ${}_{x}\mathbb{D}_\infty^{\beta/2,\lambda}$, being given as \cite[Definition 3]{Lican:15}
 \begin{eqnarray*}
&&{}_{-\infty}\mathbb{D}_x^{\beta/2,\lambda} u(x)
=\frac{e^{-\lambda x}}{\Gamma(n-\beta)}\frac{d^n}{dx^n}\int_{-\infty}^x\frac{e^{\lambda \xi}u(\xi)}{(x-\xi)^{\beta-n+1}}d\xi,
~~n=\lfloor \beta/2\rfloor +1,\nonumber\\[5pt]
&&{}_{x}\mathbb{D}_\infty^{\beta/2,\lambda} u(x)
=\frac{e^{\lambda x}}{\Gamma(n-\beta)}\frac{(-d)^n}{dx^n}\int_x^{\infty}\frac{e^{-\lambda \xi}u(\xi)}{(\xi-x)^{\beta-n+1}}d\xi,
~~n=\lfloor \beta/2\rfloor +1,
 \end{eqnarray*}
one has
  \begin{proposition}
  For $\beta\in (0,1)\cup (1,2)$, (\ref{weakformulae}) has the equivalent weak formulation:
  find $p\in  {\widetilde{H}_0^{\beta/2}(\Omega)}$ such that
 \begin{eqnarray*}
B(p,v)= (-1)^{\lfloor \beta\rfloor}B_1(p,v)=\left\langle f,v\right\rangle
 \end{eqnarray*}
 for all $v\in {\widetilde{H}_0^{\beta/2}(\Omega)}$, where  $B_1(p,v)$ is given as
 \begin{eqnarray*}
\frac{1}{2}\int_{\mathbb{R}}{}_{-\infty}\mathbb{D}_x^{\beta/2,\lambda}p\,\,{}_x\mathbb{D}_\infty^{\beta/2,\lambda}vdx
 +\frac{1}{2}\int_{\mathbb{R}}{}_x\mathbb{D}_\infty^{\beta/2,\lambda}p\,\, {}_{\infty}\mathbb{D}_x^{\beta/2,\lambda}vdx-\lambda^{\beta}\int_{\mathbb{R}}p\,vdx.
 \end{eqnarray*}
\end{proposition}

 \begin{proof}
 According to (\ref{hkkmo0dkdh34452}),  for $\beta\in (0,1)\cup(1,2)$, there exists
 \begin{eqnarray*}
 B(p,v)=\frac{(-1)^{\lfloor \beta \rfloor}}{2\pi}\int_{\mathbb{R}}\left(\frac{1}{2}(\lambda+i\xi)^{\beta}+ \frac{1}{2}(\lambda-i\xi)^{\beta}-\lambda^\beta \right)\,\mathscr{F}[p(y)](\xi)\, \overline{\mathscr{F}[v(y)](\xi)}d\xi.~~
 \end{eqnarray*}
 Then  the desired result is obtained  by using the Parseval identity (\ref{Parseval}) and \cite[Lemma 1]{Lican:15}
 \begin{eqnarray}
 &&\mathscr{F}\left[{}_{-\infty}\mathbb{D}_x^{\frac{\beta}{2},\lambda}u(x)\right]=(\lambda +i\xi)^{\beta/2}\mathscr{F}[u](\xi),\\
 &&\mathscr{F}\left[{}_x\mathbb{D}_\infty^{\frac{\beta}{2},\lambda}u(x)\right]=(\lambda -i\xi)^{\beta/2}\mathscr{F}[u](\xi).
 \end{eqnarray}
 \end{proof}

 To obtain the well-posedness of the weak formulation (\ref{weakformulae}),  we need to show that the bilinear form  $B(\cdot,\cdot)$ is  coercive, i.e.,
 \begin{eqnarray}\label{nehlllsdd=-eed}
 \left\|p\right\|^2_{H^{\beta/2}(\mathbb{R})}\lesssim B(p,p) ~~~~\forall p\in \widetilde{H}_0^{\beta/2}(\mathbb{R}).
 \end{eqnarray}
When $\lambda=0$ and $p\in \widetilde{H}_0^{\beta/2}(\Omega)$, it can be easily proved that (\ref{nehlllsdd=-eed}) holds, since
 \begin{eqnarray}
 B(p,p)\ge \int_{\mathbb{R}}\int_{\mathbb{R}\backslash\Omega}\frac{\left|p(x)\right|^2}{|x-y|^{1+\beta}}dydx\ge \frac{2(b-a)^{-\beta}}{\beta}\left\|p\right\|_{L^2(\Omega)}^2.
 \end{eqnarray}
Along this line, combining  (\ref{hkkdkdh34452}) and (\ref{hkkmo0dkdh34452}), for $\lambda>0$ and $\beta\in (0,2)$, one may expect to find a constant $C$ such that
 \begin{eqnarray}
 \mathcal{G}(\lambda,\xi,\beta)\ge C \left|\xi\right|^{\beta}
 \end{eqnarray}
for all $\xi\in \mathbb{R}$, which leads to
\begin{eqnarray}
B(p,p)\ge \frac{C}{2\pi}\int_{\mathbb{R}}\left|\xi\right|^{\beta}\left|\mathscr{F}[p(x)](\xi)\right|^2d\xi\ge C_1\left\|p\right\|^2_{H^{\beta/2}(\mathbb{R})}
\end{eqnarray}
with $C_1$ being a positive constant. Unfortunately, although
 \begin{eqnarray}\label{positivenumver}
 \mathcal{G}(\lambda,\xi,\beta)\ge0
 \end{eqnarray}
  for all $\xi \in \mathbb{R}$ (see  \ref{appendix}), by using  L'Hospital rule,  it holds that
 \begin{eqnarray}
\lim_{\xi\to 0^+} \frac{\mathcal{G}(\lambda,\xi,\beta)}{\xi^\beta}=0.
  \end{eqnarray}
  Therefore, there is no such a constant $C$.  In the following, we will work with  the bilinear form (\ref{bilinearformeddee5555}) directly.



\begin{proposition}\label{sec2proposition1}
  If $0<s<1$, then for any   real number $\delta>0$, there exists a positive constant  $C=C(\Omega, \delta,s)$ such that
 \begin{eqnarray}
 \left\|w\right\|_{L^2(\Omega)}\le C\left|w\right|_{H^s(\Omega^*)} ~~\forall w\in  {\widetilde{H}_0^{s}(\Omega)},
 \end{eqnarray}
 where $\Omega^*=\left(a-\delta,b+\delta\right)$.
 In particular, if $\frac{1}{2}<s<1$, the above result can be further improved as: there exists a positive constant $C=C(\Omega,s)$ such that
 \begin{eqnarray}\label{sec2proposition1eq2}
 \left\|w\right\|_{L^2(\Omega)}\le C\left|w\right|_{H^s(\Omega)} ~~\forall w\in  {\widetilde{H}_0^{s}(\Omega)}.
 \end{eqnarray}
 \end{proposition}
 \begin{proof}
According to the definition of the norm,
 \begin{eqnarray}
 && \left|w\right|^2_{H^s(\Omega^*)}
  =\int_{\Omega^*}\int_{\Omega^*}\frac{\left|w(x)-w(y)\right|^2}{\left|x-y\right|^{1+2s}}dy dx\nonumber\\
  &&~\ge \int_{\Omega^*}\left(\int_{a-\delta}^a \frac{\left|w(x)-w(y)\right|^2}{\left|x-y\right|^{1+2s}}dy
    +\int_b^{b+\delta}\frac{\left|w(x)-w(y)\right|^2}{\left|x-y\right|^{1+2s}}dy \right)dx\nonumber\\
  &&~=\int_{\Omega}w^2(x)\left(\int_{a-\delta}^a \frac{1}{\left|x-y\right|^{1+2s}}dy dx
    +\int_b^{b+\delta}\frac{1}{\left|x-y\right|^{1+2s}}dy \right)dx\nonumber\\
    &&=\frac{1}{2s}\int_{\Omega}w^2(x)\left(h_1(x)+h_2(x)\right)dx,
 \end{eqnarray}
 where
 \begin{eqnarray*}
 &&h_1(x)=(x-a)^{-2s}-(x-a+\delta)^{-2s},\\
 && h_2(x)=(b-x)^{-2s}-(b-x+\delta)^{-2s}.
 \end{eqnarray*}
 Note that
 \begin{eqnarray*}
 && h_1^\prime(x)=-2s\left((x-a)^{-2s-1}-(x-a+\delta)^{-2s-1}\right)<0,\\
 && h^\prime_2(x)=2s\left((b-x)^{-2s-1}-(b-x+\delta)^{-2s-1}\right)>0,
 \end{eqnarray*}
 for $x\in \Omega=(a,b)$. Then
 \begin{eqnarray}
 \left|w\right|^2_{H^s(\Omega^*)}\ge \frac{h_1(b)+h_2(a)}{2s}\left\|w\right\|^2_{L^2(\Omega)}.
 \end{eqnarray}

For (\ref{sec2proposition1eq2}), we can assume $w\in \widetilde{C}_0^\infty(\Omega)$, concluding  by density arguments.
 For any $w\in \widetilde{C}_0^\infty(\Omega)$ and $\frac{1}{2}<s<1$, we have the fractional Hardy inequality (see \cite[Theorem 2.6]{Loss:10})
  \begin{eqnarray}\label{hardyinequality2}
\int_\Omega \frac{w^2(x)}{{\rm dist}(x,\partial \Omega)^{2s}}dx \le C(\Omega,s)\int_a^b\int_a^b\frac{\left|w(x)-w(y)\right|^2}{|x-y|^{1+2s}}dydx,
\end{eqnarray}
 where ${\rm dist}(x,\partial \Omega):=\min\left\{(x-a),(b-x)\right\}$.
 By (\ref{hardyinequality2}), we have
 \begin{eqnarray}
 \left|w\right|_{H^s(\Omega)}^2&\ge &\frac{1}{C(\Omega,s)}\int_\Omega \frac{w^2(x)}{{\rm dist}(x,\partial \Omega)^{2s}}dx\nonumber\\
 &\ge&\frac{4^s}{C(\Omega,s)(b-a)^{2s}}\left\|w\right\|_{L^2(\Omega)}^2,
 \end{eqnarray}
 where ${\rm dist}(x,\partial \Omega)\le \frac{b-a}{2}$ has been used.
 \end{proof}
 \begin{remark}
For $s\in \left(0,\frac{1}{2} \right)$, because of the fractional Hardy inequality  (see \cite[Eq. 17]{Dyda:04})
 \begin{eqnarray}\label{hardyinequality1}
 ~~~~~\int_\Omega \frac{w^2(x)}{{\rm dist}(x,\partial \Omega)^{2s}}dx\le C(\Omega,s)\left(\int_\Omega\int_\Omega\frac{\left|w(x)-w(y)\right|^2}{|x-y|^{1+2s}}dydx+\left\|w\right\|^2_{L^2((a,b))}\right),
\end{eqnarray}
by using reduction to absurdity, one can show that (\ref{sec2proposition1eq2}) does not hold.
 \end{remark}

 \begin{proposition}\label{sec2proposition2}
 If $0<s<1, w\in  {\widetilde{H}_0^{s}(\Omega)}$, then for any given  real number $\delta>0$,
 \begin{eqnarray}
 \left|w\right|_{H^s(\mathbb{R})}\simeq\left|w\right|_{H^s(\Omega^*)}.
 \end{eqnarray}
 Moreover, for $\frac{1}{2}<s<1$, one actually has
 \begin{eqnarray}
 \left|w\right|_{H^s(\mathbb{R})}\simeq \left|w\right|_{H^s(\Omega)}.
 \end{eqnarray}
 \end{proposition}
 \begin{proof}
 The equivalence of $\left|w\right|_{H^s(\mathbb{R})}$ and $\left|w\right|_{H^s(\Omega^*)}$  comes from the facts that $\left|w\right|_{H^s(\Omega^*)}\le \left|w\right|_{H^s(\mathbb{R})}$ and
 \begin{eqnarray}
 \left|w\right|^2_{H^s(\mathbb{R})}&\le& \int_{\Omega^*}\int_{\Omega^*}\frac{\left|w(x)-w(y)\right|^2}{\left|x-y\right|^{1+2s}}dydx
 +2\int_{\Omega}\int_{\mathbb{R}\backslash\Omega^*}\frac{w(x)^2}{\left|x-y\right|^{1+2s}} dy dx\nonumber\\
 &\le&\left|w\right|^2_{H^s(\Omega^*)}+2\left( \int_{-\infty}^{a-\delta}\frac{1}{(a-y)^{1+2s}}dy+\int_{b+\delta}^\infty\frac{1}{(y-b)^{1+2s}}dy  \right) \left\|w\right\|^2_{L^2(\Omega)}\nonumber\\
 &\le& \left|w\right|^2_{H^s(\Omega^*)}+\frac{2\delta^{-2s}}{s}\left\|w\right\|^2_{L^2(\Omega)}\nonumber\\
 &\le& C |w|^2 _{H^s(\Omega^*)}.
 \end{eqnarray}

When $\frac{1}{2}<s<1$, the equivalence of $\left|w\right|_{H^s(\mathbb{R})}$ and $\left|w\right|_{H^s(\Omega)}$  comes from the facts that $\left|w\right|_{H^s(\Omega)}\le \left|u\right|_{H^s(\mathbb{R})}$ and
 \begin{eqnarray}
 \left|w\right|^2_{H^s(\mathbb{R})}&\le&
\int_\Omega\int_\Omega\frac{\left|w(x)-w(y)\right|^2}{\left|x-y\right|^{1+2s}} dy dx+2\int_\Omega\int_{\mathbb{R}\backslash\Omega}\frac{w^2(x)}{\left|x-y\right|^{1+2s}}dydx\nonumber\\
&\le&\left|w\right|^2_{H^s(\Omega)}+ \frac{2}{s}                                                                                           \int_\Omega \frac{w^2(x)}{{\rm dist}(x,\partial \Omega)^{2s}}dx\nonumber\\
&\le& C\left|w\right|^2_{H^s(\Omega)}.
 \end{eqnarray}
\end{proof}

 \begin{theorem}\label{Laxgramlemma1}
 The weak formulation (\ref{weakformulae}) is well-posed, and $\left\|u\right\|_{H^{\beta/2}(\mathbb{R})}\lesssim \left\|f\right\|_{H^{-\beta/2}(\Omega)}$.
 \end{theorem}
 \begin{proof}
  For any $p,v\in {\widetilde{H}_0^{\beta/2}(\Omega)}$, by the Cauchy-Schwarz inequality and $0<e^{-\lambda|y|}\le1$, we have
 \begin{eqnarray}
&& \left|B(p,v)\right|\nonumber\\
&&~~~\le \frac{c_{\beta}}{2}\left(\int_{\mathbb{R}}\int_{\mathbb{R}}\frac{\left(p(x)-p(y)\right)^2}{e^{\lambda|x-y|}|x-y|^{1+\beta}}dydx\right)^{\frac{1}{2}}
 \left(\int_{\mathbb{R}}\int_{\mathbb{R}}\frac{\left(v(x)-v(y)\right)^2}{e^{\lambda|x-y|}|x-y|^{1+\beta}}dydx\right)^{\frac{1}{2}}\nonumber\\
&&~~~\le \frac{c_{\beta}}{2}\left|p\right|_{H^{\beta/2}(\mathbb{R})}\left|v\right|_{H^{\beta/2}(\mathbb{R})}.~~~~~
\label{uueeeeeffffff1}
 \end{eqnarray}

 By Propositions \ref{sec2proposition1} and \ref{sec2proposition2}, we have
\begin{eqnarray}
 B(p,p)&\ge& \frac{c_{\beta}}{2}\int_{{\Omega}^*}\int_{{\Omega}^*}\frac{\left(p(x)-p(y)\right)^2}{e^{\lambda|x-y|}|x-y|^{1+\beta}}dydx\nonumber\\
 &\ge&\frac{c_{\beta}}{2} e^{-\lambda(b-a+2\delta)}\left|p\right|_{H^{\beta/2}(\Omega^*)}\nonumber\\
 &\ge& C\left\|p\right\|^2_{H^{\beta/2}(\mathbb{R})}.\label{uueeeeeffffff2}
 \end{eqnarray}

In addition,
 \begin{eqnarray}
 \left|\left\langle f,v\right\rangle\right|\le \left\|f\right\|_{H^{-\beta/2}(\Omega)}\left\|v\right\|_{H^{\beta/2}(\Omega)}\le \left\|f\right\|_{H^{-\beta/2}(\Omega)}\left\|v\right\|_{H^{\beta/2}(\mathbb{R})}.
 \end{eqnarray}

 Therefore, by the Lax-Milgram Theorem, the problem (\ref{weakformulae}) has an unique solution.
 \end{proof}

\section{Riesz basis Galerkin approximation}\label{sec4}
In this section, we propose the Galerkin approximation of (\ref{weakformulae}) with error analysis.
Without loss of generality, in the following, we take $\Omega:=(0,1)$.
\subsection{Single scaling B-spline and multiscale  Reisz basis functions}

To develop the numerical approximation of (\ref{weakformulae}), we need to choose the appropriate finite dimensional subspace of
 $\widetilde{H}_0^{\beta/2}(\Omega)$. Here, we use the spline wavelet spaces introduced in  \cite{Jia:09}. Let $M_m\, (m\in \mathbb{N}^+)$ be the B-spline of order $m$, i.e, for $x\in \mathbb{R}$,
\begin{eqnarray}
M_1(x)=\chi_{[0,1]}= \left\{
\begin{array}{l}
1,~~~ {\rm if}~ x \in [0,1], \\

0,~~~ {\rm otherwise},
\end{array}
 \right.~~M_m(x)=\int_0^1M_{m-1}(x-t)dt.
\end{eqnarray}
Then $M_m$ is supported on $[0,m],\,M_m>0$ for $x\in (0,m)$, and $M_m$  satisfies the following refinement equation \cite{Urban:09}
\begin{eqnarray}\label{refinerelation1}
M_m(x)=2^{1-m}\sum_{k=0}^m\binom{m}{k}M_m(2x-k).
\end{eqnarray}
Moreover, $\mathscr{F}[M_m(x)](\xi)=\left(\frac{1-e^{i\xi}}{i\xi}\right)^m\,(\xi\in\mathbb{R})$, and $M_m\in\widetilde{H}_0^{\mu}(0,m)$ for $0<\mu<m-\frac{1}{2}$. In this paper, we focus on the cases of $m=1$ and $m=2$.

Let $r=1$ or $r=2$, and $n_0$ be the least integer such that $2^{n_0}\ge 2r$. For $j\in \mathbb{Z}$, denote
\begin{eqnarray}
\phi^r_{n,j}(x):=2^{n/2}M_r(2^nx-j),~~ x\in \mathbb{R}.
\end{eqnarray}
If $n\ge n_0$ and $j\in I_n:=\left\{0,1,\ldots,2^n-r\right\}$, then $\phi^r_{n,j}(x)=0$ for $x\in \mathbb{R}/[0,1]$,
and $V_n:={\rm span}\left\{\phi^r_{n,j}: j\in I_n\right\}$ is a subspace of $\widetilde{H}_0^{\mu}(\Omega)$ for $\mu\in [0,r-\frac{1}{2})$.
 Moreover, the sequence $\left\{V_n\right\}_{n\ge n_0}$ is a multiresolution analysis (MRA) of $\widetilde{L}_2\left(\Omega\right)$, i.e.,
\begin{itemize}
\item $V_{n-1}\subset V_n$ for all $n\ge n_0$;
\item $\cup_{n=n_0}^{\infty}V_n$ is dense in $\widetilde{L}^2(\Omega)$ (in fact, by  \cite[Theorem 5]{Jia:09}, $\cup_{n=n_0}^{\infty}V_n$ also
is dense in $\widetilde{H}^{\mu}(\Omega)$ for $\mu\in (0,r-\frac{1}{2})$;
\item For all $n\ge n_0$ there exist constants $0<c_1\le c_2<\infty$ independent of $n$, such that the set
$\Phi^r_n:=\left\{\phi^r_{n,j}: j\in I_n\right\}$ forms a Riesz basis of $V_n$, i.e., for all sequences
    \begin{eqnarray*}
    {\bf d}=\left\{d_{n,0}, d_{n,1},\cdots, d_{n,2^n-r}\right\}
    \end{eqnarray*}
    we have
    \begin{eqnarray}\label{Rieszconstant}
    c_1\sum_{j\in I_n}|d_{n,j}|^2\le \big\|\sum_{j\in I_n}d_{n,j} \phi^r_{n,j}\big\|^2_{L^2(\mathbb{R})}\le c_2\sum_{j\in I_n}|d_{n,j}|^2.
    \end{eqnarray}
    \end{itemize}

For $n\ge n_0$, the nest property of $V_n$ allows one to construct the spaces $W_n:=V_{n+1}\cap V_n^{\perp}$
satisfying $V_{n+1}=V_n\oplus W_n$. More precisely, let $J_n:=\left\{1,\cdots,2^n\right\}$; for $r=1$,  defining
\begin{eqnarray}\label{refinerelation2}
\psi(x):=\frac{1}{2}\left(M_1(2x)-M_1(2x-1)\right),~~~\psi^1_{n,j}(x)=2^{n/2}\psi(2^nx-j+1),
\end{eqnarray}
and $\Psi^1_n=\left\{\psi^1_{n,j}(x),\, j\in J_n\right\}$, then $W_n={\rm span}\left\{\Psi^1_n\right\}$; for $r=2$, defining
\begin{eqnarray}
&&\psi(x)=\frac{1}{24}M_2(2x)-\frac{1}{4}M_2(2x-1)+\frac{5}{12}M_2(2x-2)-\frac{1}{4}M_2(2x-3)\nonumber\\
&&~~~~~~~~~~~+\frac{1}{24}M_2(2x-4),\label{refinerelation3}\\
&&\psi_1(x)=\frac{3}{8}M_2(2x)-\frac{1}{4}M_2(2x-1)+\frac{1}{24}M_2(2x-2),\label{refinerelation4}\\
&&\psi^2_{n,j}(x)=
\left\{
\begin{array}{lll}
2^{n/2}\psi_1(2^nx),& j=1,\\2^{n/2}\psi(2^nx-j+2),&j=2,\cdots,2^n-1,\\2^{n/2}\psi_1(2^n(1-x)),&j=2^n,
\end{array}
\right.\label{refinerelation5}
\end{eqnarray}
and $\Psi^2_n=\left\{\psi^2_{n,j}(x),\, j\in J_n\right\}$, then $W_n={\rm span}\left\{\Psi^2_n\right\}$.
\begin{remark}\label{remarkprojection}
Here, the cases  for $r=1$  and $r=2$ are obtained by letting $r=s=1$ and $r=s=2$ in \cite[pp. 179-181]{Jia:09}, respectively.
\end{remark}
Because of the property of MRA, we have $\widetilde {L}^2(\Omega)=V_{n_0}\oplus \oplus_{n=n_0}^{\infty} W_n$.
Therefore, $\Phi_{n_0}\cup\cup_{n_0}^\infty \Psi_{n}$ is a new basis of $\widetilde{L}^2(\Omega)$, called multiscale basis.
\begin{lemma}[Theorems $1$ and $2$ of \cite{Jia:09}]\label{Riezddddfff}
For $n\ge n_0$, and $j\in J_n$, let $\psi^r_{n,j}$ be the functions as constructed above. Then
\begin{eqnarray} \label{reizedewddd1}
\left\{2^{-n_0\mu}\phi^r_{n_0,j}: j\in I_{n_0}\right\}\cup \cup_{n=n_0}^\infty\left\{2^{-n\mu}\psi^r_{n,j}: j\in J_n\right\}
\end{eqnarray}
forms a Riesz basis of $\widetilde{H}_0^{\mu}(\Omega)$ for $0\le \mu<r-\frac{1}{2}$.
\end{lemma}
By Lemma \ref{Riezddddfff}, for $\mu \in [0,r-\frac{1}{2})$,  it holds that
 \begin{eqnarray}
 \left\|\phi_{n_0,j}^r\right\|_{H^{\mu}(\mathbb{R})}\simeq 2^{{n_0}\mu},~~j\in I_n; ~~ \left\|\psi_{n,j}^r\right\|_{H^{\mu}(\mathbb{R})}\simeq 2^{n\mu},~~j\in J_n,~ n\ge n_0.
 \end{eqnarray}
  Then for $n\ge n_0$, by (\ref{uueeeeeffffff1}) and (\ref{uueeeeeffffff2}), we know that the set
\begin{eqnarray} \label{reizedewddd2}
 \begin{array}{l}
 \left\{\frac{\phi^r_{n_0,j}}{\sqrt{B(\phi^r_{n_0,0},\phi^r_{n_0,0})}}: j\in I_{n_0}\right\}\\
~\bigcup \bigcup_{l=n_0}^\infty\left\{\frac{\psi^r_{l,1}}{\sqrt{B(\psi^r_{l,1},\psi^r_{l,1})}},
\frac{\psi^r_{l,j}}{\sqrt{B(\psi^r_{l,2},\psi^r_{l,2})}}\left|_{j=2}^{2^l-1}\right., \frac{\psi^r_{l,2^l}}{\sqrt{B(\psi^r_{l,1},\psi^r_{l,1})}}\right\}
\end{array}
\end{eqnarray}
also forms a Riesz basis of $\widetilde{H}^{\beta/2}_0(\Omega)$. Here $\beta\in (0,1)$ for $r=1$, and $\beta\in (0,2)$ for $r=2$.

We take the subspace $V_n$ as the approximation space of $\widetilde{H}_0^{\beta/2}(\Omega)$, that is, find $p_n\in V_n$ such that
\begin{eqnarray}\label{approximationsolution}
B(p_n,v_n)=(f,v_n) ~~~\forall v_n\in V_n.
\end{eqnarray}
Note that the space $V_n$ generated by $M_1(x)$  is a subspace of $\widetilde{H}^{\beta/2}_0(\Omega)$  only for $0<\beta<1$.

\subsection{Convergence analysis}
 In the following, we give convergence analysis.

\begin{proposition}\label{section4lemma1}
 For  $w\in H^\alpha(\Omega)\cap  \widetilde{H}_0^{\beta/2}(\Omega) \,(\alpha\ge\beta/2)$ and the orthogonal projection operator $P_n$ from $\widetilde{L}^2(\Omega)$ to $V_n$,  it holds that
 \begin{eqnarray}
 \left\|w-P_nw\right\|_{H^{\beta/2}(\mathbb{R})}\lesssim 2^{-n\left(\alpha-\beta/2\right)}\left\|w\right\|_{H^\alpha(\Omega)},
 \end{eqnarray}
 where $0<\beta<1$ and $\beta/2\le\alpha\le1$ if $V_n$ is generated from $M_1(x)$, and $0<\beta<2$ and $\beta/2\le\alpha\le2$ if $V_n$ is generated from $M_2(x)$.
 \end{proposition}
 \begin{proof}
For $w\in \widetilde{L}^2(\Omega)$ and $n\ge n_0$, let $P_nf$ be the  orthogonal projection from $\widetilde{L}^2(\Omega)$ to $V_n$, i.e.,
\begin{eqnarray}
\langle P_n w, \phi^{r}_{n,k}\rangle=\langle w, \phi^{r}_{n,k}\rangle ~~~\forall k\in I_n.
\end{eqnarray}
By Remark \ref{remarkprojection}, it is easily seen that $P_n$  actually is a special case of the projector $P_n$ defined in \cite[pp. 197]{Jia:09}.
Then $\left\|P_n\right\|:={\rm sup}_{w\in \widetilde{L}^2(\Omega), \left\|w\right\|_{L^2(\mathbb{R})}\le 1}\left\|P_nw\right\|_{L^2(\mathbb{R})}$
is bounded by a constant independent of $n$; $P_{n+1}w-P_nw$ lies in $V_{n+1}\cap V_n^{\perp}=W_n$; and $\lim_{n\to \infty}\left\|P_nw-w\right\|_{L^2(\mathbb{R})}$ $=0$.
Combining with Lemma \ref{Riezddddfff}, for any $w\in \widetilde{L}^2(\Omega)$, we have
\begin{eqnarray}
~~~w=P_{n_0}w+\sum_{n=n_0}^{\infty} \left(P_{n+1}w-P_nw\right)=\sum_{j\in I_{n_0}}d_{n_0,j}\phi^r_{n_0,j}
+\sum_{n=n_0}^\infty\sum_{j\in J_n} c_{n,j}\psi^r_{n,j},~~
\end{eqnarray}
and
\begin{eqnarray}\label{normequaeeett}
\left\|w\right\|^2_{H^\mu(\mathbb{R})}&\simeq& \sum_{j\in I_{n_0}}\left|2^{n_0\mu}d_{n_0,j}\right|^2
+\sum_{n=n_0}^{\infty}\sum_{j\in J_n}\left|2^{n\mu}c_{n,j}\right|^2\nonumber\\
&\simeq& 2^{2n_0\mu}\left\|P_{n_0}w\right\|^2_{L^2(\mathbb{R})}
+\sum_{n= n_0}^{\infty} 2^{2n\mu}\left\|(P_{n+1}-P_n)w\right\|_{L^2(\mathbb{R})}^2
\end{eqnarray}
if $w\in \widetilde{H}_0^{\mu}(\Omega),\, \mu\in [0,r-\frac{1}{2})$ further.

Firstly, it is easy to check that $P_lP_n=P_l$ for all $n_0\le l\le n$. Thus, for $w_n\in V_n$ , we have $w_n=\sum_{l=n_0}^{n}\left(P_{l}-P_{l-1}\right)w_n$ with $P_{n_0-1}:=0$.  By (\ref{normequaeeett}) and the uniform boundedness of $\left\|P_n\right\|$, for $\mu\in [0,r-\frac{1}{2})$, it holds that
 \begin{eqnarray}\label{inverestimatieerd}
&&\left\|w_n\right\|^2_{H^{\mu}(\mathbb{R})}\lesssim \sum_{l=n_0}^n 2^{2l\mu}\left\|\left(P_{l}-P_{l-1}\right)w_n\right\|^2_{L^2(\mathbb{R})}\nonumber\\
&&~~~~~~~~~~~~~~\lesssim\left(\sum_{l=n_0}^n 2^{2l\mu}\right)\left\|w_n\right\|_{L^2(\mathbb{R})}^2\lesssim 2^{2n\mu}\left\|w_n\right\|_{L^2(\Omega)}^2.
\end{eqnarray}

Secondly, for $r=1$,
we have \cite[pp. 13-16]{Urban:09}
\begin{eqnarray}\label{directestimate1}
\left\|w-P_n w\right\|_{L^2(\Omega)}\lesssim 2^{-n\alpha}\left\|w\right\|_{H^\alpha(\Omega)},~~0\le \alpha\le 1;
\end{eqnarray}
for $r=2$, since $V_n\left|_{\Omega}\right.$ actually is the space $S_j$  with $d=2$ in \cite[Lemma 5]{Zhang:17}, we have
\begin{eqnarray}\label{directestimate2}
&&\left\|w-P_nw\right\|_{L^2(\Omega)}\le(1+\sup_{l\ge n_0}\left\|P_l\right\|)\inf_{g\in V_n\left|_{\Omega}\right.} \left\|w-g\right\|_{L^2(\Omega)}\nonumber\\
&&~~~~~~~~~~~~~~~~~~~~~\lesssim 2^{-n\alpha}\left\|w\right\|_{H^\alpha(\Omega)},~~0\le \alpha\le 2.
\end{eqnarray}

Finally,  for any $w\in H^\alpha(\Omega)\cap  \widetilde{H}_0^{\beta/2}(\Omega) \,(\alpha\ge\beta/2)$, it holds that
$w=P_nw+\sum_{l\ge n} \left(P_{l+1}-P_l\right)w;$
by (\ref{inverestimatieerd}), (\ref{directestimate1}), and (\ref{directestimate2}), we have
\begin{eqnarray}
&&\left\|w-P_nw\right\|_{H^{\beta/2}(\mathbb{R})}\le \sum_{l\ge n}\left\|P_{l+1}w-P_l w\right\|_{H^{\beta/2}(\mathbb{R})}\nonumber\\
&&~~~\lesssim  \sum_{l\ge n}2^{\beta l/2}\left(\left\|P_{l+1}w-w\right\|_{L^2(\Omega)}+\left\|w-P_l w\right\|_{L^2(\Omega)}\right)\nonumber\\
&&~~~\lesssim   \sum_{l\ge n}2^{(\beta /2-\alpha)l}\left\|w\right\|_{H^\alpha(\Omega)}\lesssim 2^{-n(\alpha-\beta/2)}\left\|w \right\|_{H^\alpha(\Omega)}.
\end{eqnarray}
Thus, we complete the proof.
\end{proof}

\begin{theorem}\label{errorestimate}
 Let  $p\in H^\mu(\Omega)\cap \widetilde{H}_0^{\beta/2}(\Omega) \,(\mu\ge\beta/2)$ be the exact solution of  (\ref{weakformulae}) and $p_n\in V_n$ be the  approximation solution of (\ref{approximationsolution}). Then
 \begin{eqnarray}
 \left\|p-p_n\right\|_{H^{\beta/2}(\mathbb{R})}\lesssim 2^{-n\left(\min\{\mu,r\}-\beta/2\right)}\left\|p\right\|_{H^\mu(\Omega)},
\end{eqnarray}
where $\beta\in (0,1)$ if $V_n$ is generated from $M_1(x)$, and  $\beta\in (0,2)$ if $V_n$ is generated from $M_2(x)$.
 \end{theorem}
  \begin{proof}
  Using the standard argument technique for   C{\'e}a's lemma (see Theorem (2.8.1) of \cite{Brenner:02}), we have
 \begin{eqnarray}
 \left\|p-p_n\right\|_{H^{\beta/2}(\mathbb{R})}\lesssim\inf_{v\in V_n}\left\|p-v\right\|_{H^{{\beta}/{2}}(\mathbb{R})}.
\end{eqnarray}
Then the desired result is a direct conclusion of  Proposition \ref{section4lemma1}.
\end{proof}

\section{Implementation details}\label{sec5}
 It is easy to check that $V_n=V_{n_0}\oplus \oplus_{j=n_0}^{n-1} W_j$. Then  $V_n$ has two types of  basis functions:
the single scaling B-spline basis functions ${\Psi}^r_n$ and the multiscale Reisz basis functions
\begin{eqnarray}\label{multiscalebases2}
 \begin{array}{l}
\widetilde{\Psi}^r_n:= \left\{\frac{\phi^r_{n_0,j}}{\sqrt{B(\phi^r_{n_0,0},\phi^r_{n_0,0})}}: j\in I_{n_0}\right\}\\
~~~~~~~~\bigcup \bigcup_{l=n_0}^{n-1}\left\{\frac{\psi^r_{l,1}}{\sqrt{B(\psi^r_{l,1},\psi^r_{l,1})}},
\frac{\psi^r_{l,j}}{\sqrt{B(\psi^r_{l,2},\psi^r_{l,2})}}\left|_{j=2}^{2^l-1}\right., \frac{\psi^r_{l,2^l}}{\sqrt{B(\psi^r_{l,1},\psi^r_{l,1})}}\right\}.
\end{array}
\end{eqnarray}

\subsection{Computing the stiffness matrix}
We  first consider the stiffness matrix ${\bf A}:=B\left(\Phi_n^r,\Phi_n^r\right)$ of single scaling basis functions.
Making use of the fact that $\Phi_n^r$ are obtained from the translations of a single function $M_r(x)$, we have
\begin{proposition}\label{Teoplitzstruch}
 $ B(\Phi^r_n,\Phi^r_n)$  is a symmetric Toeplitz matrix.
 \end{proposition}

 \begin{proof}
 Since
 \begin{eqnarray}
&& \mathscr{F}\left[\phi^r_{n, j_1}\right](\xi)=2^{-\frac{n}{2}}e^{-{i\frac{j_1}{2^n}}\xi}\mathscr{F}\left[M_r(x)\right]\left(\frac{\xi}{2^n}\right),\nonumber\\
&&\mathscr{F}\left[\phi^r_{n, j_2}\right](\xi)=2^{-\frac{n}{2}}e^{-{i\frac{j_2}{2^n}}\xi}\mathscr{F}\left[M_r(x)\right]\left(\frac{\xi}{2^n}\right),\nonumber
 \end{eqnarray}
 for $\lambda=0$, by (\ref{hkkdkdh34452}) we have
 \begin{eqnarray}
 &&B\left(\phi^r_{n, j_1},\phi^r_{n,j_2}\right)=\int_{\mathbb{R}}|\xi|^{\beta} \mathscr{F}[\phi^r_{n,j_1}(x)](\xi)\,\overline{\mathscr{F}[\phi^r_{n,j_2}(x)](\xi)}d\xi\nonumber\\
 &&~~~~=\frac{1}{2^n}\int_{\mathbb{R}}|\xi|^\beta e^{i \frac{
 j_2-j_1}{2^n}\xi}\left|\mathscr{F}[M_r(x)]\left(\frac{\xi}{2^n}\right)\right|^2d\xi\nonumber\\
 &&~~~~=\frac{2^{r+1}}{2^n}\int_0^{\infty} \xi^\beta \cos\left(\frac{j_2-j_1}{2^n}\xi\right)\left(\frac{1-\cos(\xi/2^n)}{(\xi/2^n)^2}\right)^r d\xi;
 \end{eqnarray}
similarly, for $\lambda>0$, by  (\ref{hkkmo0dkdh34452}) we have
 \begin{eqnarray}
 &&B\left(\phi^r_{n, j_1},\phi^r_{n,j_2}\right)=\int_{\mathbb{R}} \mathcal{G}(\lambda,\xi,\beta) e^{i \frac{
 j_2-j_1}{2^n}\xi}\left|\mathscr{F}[M_r(x)]\left(\frac{\xi}{2^n}\right)\right|^2d\xi\nonumber\\
 &&~~~~~=\frac{2^{r+1}}{2^n}\int_0^{\infty} \mathcal{G}(\lambda,\xi,\beta)\cos\left(\frac{j_2-j_1}{2^n}\xi\right)\left(\frac{1-\cos(\xi/2^n)}{(\xi/2^n)^2}\right)^r d\xi,
 \end{eqnarray}
in the last step the fact that $\mathcal{G}(\lambda,\xi,\beta)$  is an even function w.r.t. $\xi$ has been used.
The desired result follows from that $B\left(\phi^r_{n, j_1},\phi^r_{n,j_2}\right)$  remains constant if $ |j_1-j_2|$ is a constant.
\end{proof}

 Therefore, we only need to calculate and store the first row of matrix $B(\Phi_n^r,\Phi_n^r)$. Let $h=2^{-n}, {\zeta(y)}=e^{-\lambda y} y^{-1-\beta}$, and $\kappa=\frac{h^3}{c_\beta}$. By using  the Fubini theorem, for the entries of $B(\Phi_n^1,\Phi_n^1)$ we have
 \begin{eqnarray}
 &&B(\Phi_n^1,\Phi_n^1)_{0,0}=\frac{2c_\beta}{h}\int_0^h y{\zeta(y)}dy+2c_\beta\int_h^\infty {\zeta(y)}dy,\label{elementM_1}\\
 && B(\Phi_n^1,\Phi_n^1)_{0,j}=\frac{c_\beta}{h}\int_{(j-1)h}^{jh}\left((j-1)h-y\right) {\zeta(y)}dy\nonumber\\
 &&~~~~~~~~~~~~~+\frac{c_\beta}{h}\int_{jh}^{(j+1)h}\left(y-(j+1)h\right){\zeta(y)}dy,
 \end{eqnarray}
  where $j=1,2,\cdots,2^n-1$; for the entries of $B(\Phi_n^2,\Phi_n^2)$ we have
 \begin{eqnarray}
 &&\kappa B(\Phi_n^2,\Phi_n^2)_{0,0}=\int_0^h (2h-y)y^{2} {\zeta(y)}dy+\frac{4h^3}{3}\int_{2h}^{\infty} {\zeta(y)} dy\nonumber\\
 &&~~~~~~~~~~~~+\int_h^{2h}\left(\frac{y^3}{3}-2hy^2+4h^2y-\frac{4h^3}{3}\right){\zeta(y)}dy,\label{elementM_2}\\
 &&\kappa B(\Phi_n^2,\Phi_n^2)_{0,1}=\int_0^h\frac{2y-3h}{3}y^2{\zeta(y)}dy+\int_{2h}^{3h} \frac{(y-3h)^3}{6} {\zeta(y)}dy\nonumber\\
 &&~~~~-\int_h^{2h} \left(\frac{y^3}{2}-\frac{5hy^2}{2}+\frac{7h^2y}{2}-\frac{7h^3}{6}\right){\zeta(y)} dy+\frac{h^3}{3} \int_{2h}^\infty {\zeta(y)}dy,\\
 && \kappa B(\Phi_n^2,\Phi_n^2)_{0,j}=\int_{(j-2)h}^{(j-1)h}\frac{\left((j-2)h-y\right)^3}{6}{\zeta(y)}dy+\frac{1}{6}\int_{(j-1)h}^{jh} b_1(y) {\zeta(y)} dy\nonumber\\
 &&~~~~~~+\frac{1}{6}\int_{jh}^{(j+1)h} b_2(y) {\zeta(y)} dy+\int_{(j+1)h}^{(j+2)h}\frac{\left(y-(j+2)h\right)^3}{6} {\zeta(y)}dy\label{elementM_1rd}
 \end{eqnarray}
for $j=2,3,\cdots, 2^n-2$, where
 \begin{eqnarray*}
 &&b_1(y)=3y^3+\left(6-9j\right)hy^2+3j\left(3j-4\right)h^2y-\left(3j^3-6j^2+4\right)h^3,\\
 &&b_2(y)=-3y^3+(6+9j)hy^2-3j\left(3j+4\right)h^2y+(3j^3+6j^2-4)h^3.
 \end{eqnarray*}
If $\lambda=0$, all the integrals above can be calculated exactly. When $\lambda\not=0$, we can calculate  them numerically with
 some  regularization techniques. For example, for $\beta\not=1$, we can first rewrite  $ \int_0^h y{\zeta(y)}dy$ as
 \begin{eqnarray}\label{techeneqieucalsle}
{\Gamma(1-\beta)}\left(\sum_{l=1}^{K-1} \frac{e^{-\lambda h}\lambda^{l-1}h^{l-\beta}}{\Gamma(l+1-\beta)}+\frac{\lambda^{K-1}}{\Gamma(K-\beta)}\int_{0}^h
 e^{-\lambda y} y^{K-\beta}dy\right),
 \end{eqnarray}
  and then calculate
  \begin{eqnarray}
\int_{0}^h e^{-\lambda y} y^{K-\beta}dy=\left(\frac{h}{2}\right)^{K-\beta+1}\int_{-1}^1e^{-\frac{\lambda h}{2}(1+\eta)}
\left(1+\eta\right)^{K-\beta}d\eta
 \end{eqnarray}
 by the Gauss-Jacobi quadrature with the  weight function $(1-\eta)^0(1+\eta)^{K-\beta}$ \cite{Shen:11}; we first rewrite $ \int_{h}^{\infty}\zeta(y) dy$ as
 \begin{eqnarray*}
\sum_{l=1}^2\frac{-\Gamma(-\beta)}{e^{\lambda h}h^{\beta}}\frac{(\lambda h)^{l-1} }{\Gamma(l-\beta)}+\lambda^2\Gamma(-\beta)
-\frac{\lambda ^2\Gamma(-\beta)}{\Gamma(2-\beta)}\int_0^h y^{1-\beta}e^{-\lambda y} dy,~~~~
 \end{eqnarray*}
 and then calculate $\int_0^h y^{1-\beta}e^{-\lambda y} dy$ with the techniques  similar to (\ref{techeneqieucalsle}).
 For $\beta=1$, we can first rewrite $ \int_{2h}^\infty\zeta(y)dy$ as
 \begin{eqnarray}
 \frac{e^{-2h\lambda}}{2h}-\lambda \int_{2h}^{\infty}e^{-y}y^{-1}dy
 \end{eqnarray}
 and then calculate the exponential integra1 $\int_{2h}^{\infty}e^{-y}y^{-1}dy$ with  the series expansion representation
 in \cite[Eq. 5.1.11]{Abramowitz:65}.

\subsection{Condition number and preconditioning}
This subsection focuses on reducing the condition number by using the multiscale basis.
 \begin{proposition}
The condition number of ${\bf A}$  satisfies ${\rm Cond_2}({\bf A})\simeq \, 2^{n\beta}$.
\end{proposition}
  \begin{proof}
 Let
  $\lambda_{\max}({\bf A})$ and  $\lambda_{\min}({\bf A})$ be the maximal and minimal eigenvalues of ${\bf A}$, respectively.
  Then ${\rm Cond_2}({\bf A})=\frac{\lambda_{\max}(\bf{A})}{\lambda_{\min}(\bf{A}}$. Let  ${\bf d}_n=\left(d_{n,0}, d_{n,1}, \cdots, d_{n,2^n-r}\right)^{\rm T}$. By Theorem 1.2 of \cite{Chan:07}, it holds that
  \begin{eqnarray}
  \lambda_{\min}(\bf{A})=\inf_{ {\bf d}_n\not={\bf 0}}\frac{({\bf d}_n, {\bf A} {\bf d}_n)}{\left({\bf d}_n,{\bf d}_n\right)},~~~
    \lambda_{\max}(\bf{A})=\sup_{ {\bf d}_n\not={\bf 0}}\frac{({\bf d}_n, {\bf A} {\bf d}_n)}{\left({\bf d}_n,{\bf d}_n\right)}.
\end{eqnarray}

 Firstly, because of Theorem \ref{Laxgramlemma1}, it holds that ${\bf d}_n^{\rm T} {\bf A} {\bf d}_n=B(\Phi_n^r{\bf d}_n,\Phi_n^r{\bf d}_n)
 \simeq \left\|\Phi_n^r{\bf d}_n\right\|^2_{H^{\beta/2}(\mathbb{R})}$.
  By (\ref{Rieszconstant}) and (\ref{inverestimatieerd}),  we have
\begin{eqnarray}
  1\le \frac{B(\Phi_n^r{\bf d}_n,\Phi_n^r{\bf d}_n)}{{\bf d}_n^{\rm T} {\bf d}_n}\simeq \frac{\|\Phi_n^r{\bf d}_n\|^2_{H^{\beta/2}(\mathbb{R})}}
  {\|\Phi_n^r{\bf d}_n\|^2_{L^2(\mathbb{R})}}\le 2^{n\beta}.
 \end{eqnarray}
 Hence, ${\rm Cond_2}({\bf A})\lesssim 2^{n\beta}$ for $r=1$ and $r=2$.

 Secondly, for the $V_n$ generated from $M_1(x)$ and $0<\beta<1$, by  (\ref{elementM_1})
\begin{eqnarray}
~~~~~~~~~\lambda_{\max}({\bf A})\ge B\left(\phi^1_{n,0},\phi^1_{n,0}\right)\ge\frac{c_\beta}{h}\int_0^h y{\zeta(y)}dy\ge \frac{c_\beta e^{-\lambda}}{h}\int_0^h y^{1-\beta}dy=\frac{c_\beta e^{-\lambda} }{1-\beta} h^{-\beta}.
\end{eqnarray}
In addition,  $1=\Phi^r_{n}{\bf d}\in V_n$ with ${\bf d}=(\frac{1}{2^{n/2}},\frac{1}{2^{n/2}},\cdots,\frac{1}{2^{n/2}})^{\rm T} $ and
\begin{eqnarray}
B(1,1)\le2\int_\Omega\int_{\mathbb{R}\backslash \Omega}\frac{1}{|x-y|^{1+\beta}}dydx=\frac{4}{\beta(1-\beta)}.
\end{eqnarray}
Thus, $\lambda_{\min}({\bf A})\le  \frac{B(1,1)}{{\bf d}^{\rm T}{\bf d}}\lesssim \frac{4}{\beta(1-\beta)}$.
Hence, $h^{-\beta}=2^{n\beta}\lesssim{\rm Cond}_2({\bf A})$.

Finally, for the $V_n$ generated from $M_2(x)$, by  (\ref{elementM_2}), we have
\begin{eqnarray}
&&\lambda_{\max}({\bf A})\ge B\left(\phi^2_{n,0},\phi^2_{n,0}\right)\ge \frac{1}{\kappa}\int_0^h (2h-y)y^{2} {\zeta(y)}dy\nonumber\\
&&~~~~~~~~~~~\ge \frac{he^{-\lambda}}{\kappa}\int_0^h y^{1-\beta}dy=\frac{c_\beta e^{-\lambda}}{2(2-\beta)}h^{-\beta},
\end{eqnarray}
where $\frac{y^3}{3}-2hy^2+4h^2y-\frac{4h^3}{3}>0$ with $y\in [h,2h]$ has been used in the second step.
In addition,  note that
\begin{eqnarray}
\theta(x)=\left\{\begin{array}{ll}2x& x\in [0,\frac{1}{2}]\\ 2-2x& x\in (\frac{1}{2},1]\\0& {\rm ~~~else}\end{array}\right.
=\sum_{j\in I_n}\frac{\theta((j+1)/2^n)}{2^{n/2}}\phi^2_{n,j}\in V_n;
\end{eqnarray}
 and by proposition \ref{sec2proposition1}, we have
\begin{eqnarray}
B\left(\theta(x),\theta(x)\right)\lesssim \left|\theta(x)\right|_{H^{\beta/2}(\Omega^*)}\lesssim \left\|\theta(x)\right\|_{H^1(\Omega^*)}\lesssim 1.
\end{eqnarray}
Let
\begin{eqnarray}
{\bf d}=\frac{1}{2^n}\left( \theta \left(\frac{1}{2^n}\right), \theta \left(\frac{2}{2^n}\right),\cdots, \theta\left(\frac{2^n-1}{2^n}\right)\right)^T.
 \end{eqnarray}
 By (\ref{Rieszconstant}), it holds that ${\bf d}^{\rm T}{\bf d} \lesssim 1$. Thus, $\lambda_{\min}({\bf A})\le \frac{B(\theta(x),
 \theta(x))}{{\bf d}^{\rm T}{\bf d}}\lesssim 1$. Hence, $2^{n\beta}\lesssim {\rm Cond}_2({\bf A})$.
\end{proof}

In the following, we consider  the stiffness matrix $\widetilde{\bf B}:=B\left(\widetilde{\Psi}^r_n,\widetilde{\Psi}^r_n \right) $.
 \begin{proposition}
The condition number of  $\widetilde{\bf B}$  satisfies ${\rm Cond_2}(\widetilde{\bf B})\lesssim 1$.
\end{proposition}
\begin{proof} Let
$$\widetilde {{\bf c}}_n:=\left({\bf{d}}_{n_0}^{\rm T}, {\bf{c}}_{n_0}^{\rm T},{\bf{c}}_{n_0+1}^{\rm T}, \cdots,{{\bf c}}_l^{\rm T},
\cdots, {\bf c}_{n-1}^{\rm T}\right)^{\rm T}$$
with ${\bf c}_l=\left(c_{l,1}, c_{l,2},\cdots, c_{l,2^l}\right)^{\rm T}$ for $l=n_0,\cdots,n-1$.
By Theorem \ref{Laxgramlemma1} we have
$${\widetilde {{\bf c}}_n}^{\rm T} \widetilde{\bf B} {\widetilde {{\bf c}}_n}
=B(\widetilde{\Psi}^r_n{\widetilde {{\bf c}}_n},\widetilde{\Psi}^r_n{\widetilde {{\bf c}}_n})
\simeq \left\|\widetilde{\Psi}^r_n{\widetilde {{\bf c}}_n}\right\|^2_{H^{\beta/2}(\mathbb{R})}\simeq\widetilde {{\bf c}}_n^{\rm T}\widetilde {{\bf c}}_n,$$
where the last term comes from the fact that  (\ref{reizedewddd2}) forms a
Riesz bases of $\tilde{H}_0^{\beta/2}(\Omega)$. Thus, we complete the proof.
\end{proof}

 Since $\widetilde{\Psi}^r_n$ are not composed of translations of a single function, the result
 like Proposition \ref{Teoplitzstruch} does not hold again.
 However, since both $\Phi_n^r$ and $ \widetilde{\Psi}^r_n$ are  the basis functions of $V_n$, there exists a matrix $\widetilde{{\bf M}}^r_n$ such that
 \begin{eqnarray}\label{waveletbasissmatrix1}
 \widetilde{\Psi}^r_n=\Phi_n^r \widetilde{{\bf M}}^r_n.
 \end{eqnarray}
 Then
$\widetilde {\bf B}=\left( \Phi_n^r \widetilde{{\bf M}}^r_n,\Phi_n^r \widetilde{{\bf M}}^r_n \right)
=\left(\widetilde{{\bf M}}^r_n\right)^{\rm T}B(\Phi_n^r,\Phi_n^r) \widetilde{{\bf M}}^r_n
=\left(\widetilde{{\bf M}}^r_n\right)^{\rm T} {\bf A}\widetilde{{\bf M}}^r_n.
$
 To obtain  $\widetilde{{\bf M}}^r_n$, for $l\ge n_0$,   by (\ref{refinerelation1}), and (\ref{refinerelation2})-(\ref{refinerelation5}), there exist matrices
${\bf  M}^r_{l,0}$ and ${\bf M}^r_{l,1}$ such that
 \begin{eqnarray}
 \Phi_l^r=\Phi_{l+1}^r {\bf  M}^r_{l,0},~~~\Psi_l^r=\Phi_{l+1}^r {\bf M}^r_{l,1}.
 \end{eqnarray}
 Denote ${\bf M}_l^r=({\bf M}_{l,0}^r,{\bf M}_{l,1}^r)$. We have
 \begin{eqnarray}
 \left(\Phi^r_{n_0}, \Psi^r_{n_0}, \cdots, \Psi_{n-2}^r,\Psi_{n-1}^r\right)=\Phi_{n}^r {\bf M}^r,
 \end{eqnarray}
 \begin{eqnarray}
 {\bf M}^r={\bf M}_{n-1}^r\left(\begin{array}{ll}{\bf M}_{n-2}^r&{\bf 0}\\{\bf 0}&{\bf I}^r_{n-2}\end{array}\right)
 \left(\begin{array}{ll}{\bf M}_{n-3}^r&{\bf 0}\\{\bf 0}&{\bf I}^r_{n-3}\end{array}\right)\cdots
 \left(\begin{array}{ll}{\bf M}_{n_0}^r&{\bf 0}\\{\bf 0}&{\bf I}^r_{n_0}\end{array}\right)
 \end{eqnarray}
with ${\bf I}_{l}^r, \, l=n_0,n_1,\cdots, n-1$ being identity matrices. Define a diagonal matrix $\widetilde{{\bf D}}^r_n$ as
\begin{eqnarray*}
&&{\rm diag}\Big(\underbrace{a^r_{n_0},\cdots, a^r_{n_0}}_{2^{n_0-r+1}}, b^r_{n_0,1},\underbrace{b^r_{n_0,2},\cdots, b^r_{n_0,2}}_{2^{n_0}-2} b^r_{n_0,1},\nonumber\\
&&~~~b^r_{n_0+1,1},\underbrace{b^r_{n_0+1,2},\cdots, b^r_{n_0+1,2}}_{2^{n_0+1}-2} b^r_{n_0+1,1},\cdots,
 b^r_{n-1,1},\underbrace{b^r_{n-1,2},\cdots, b^r_{n-1,2}}_{2^{n-1}-2} b^r_{n-1,1}\Big)~~~~~~~~
\end{eqnarray*}
with $a^r_{n_0}={{B(\phi_{n_0,1}^r,\phi_{n_0,1}^r)^{-\frac{1}{2}}}}$,
 and $b^r_{l,1}={{B(\psi_{l,1}^r,\psi_{l,1}^r)^{-\frac{1}{2}}}}, \,b^r_{l,2}={{B(\psi_{l,2}^r,\psi_{l,2}^r)^{-\frac{1}{2}}}}$
for $l=n_0, n_0+1,\cdots, n-1$. Then from  (\ref{multiscalebases2}), we have $\widetilde{{\bf M}}^r_n={\bf M}^r\widetilde{{\bf D}}^r_n$. Note that
$ a^r_{n_0}$ and $b^r_{l,1}, \,b^r_{l,2}$ can be calculated  by the relations  (\ref{refinerelation2})-(\ref{refinerelation5}), Proposition \ref{Teoplitzstruch}, and
the formulae (\ref{elementM_1})-(\ref{elementM_1rd}). For example, $B(\psi_{l,2}^1,\psi_{l,2}^1)$ can be obtained by
\begin{eqnarray}
\frac{1}{\sqrt{2}}\left(\frac{1}{2},-\frac{1}{2}\right)
\left(\begin{array}{cc} B\left(\Phi^1_{l+1},\Phi^1_{l+1}\right)_{0,0}&B\left(\Phi^1_{l+1},\Phi^1_{l+1}\right)_{0,1}\\[4pt]
B\left(\Phi^1_{l+1},\Phi^1_{l+1}\right)_{0,1}&B\left(\Phi^1_{l+1},\Phi^1_{l+1}\right)_{0,0}\end{array}\right)\nonumber
\frac{1}{\sqrt{2}}\left(\frac{1}{2},-\frac{1}{2}\right)^{\rm T}.
\end{eqnarray}

In practice, we do not need to generate the stiffness matrix $\widetilde {\bf B}$ explicitly;
the purpose of introducing  the multiscale basis functions usually is to obtain  the preconditioning matrix of ${\bf A}$,
due to its density and the increasing condition number.
 Let $p_n=\Phi_n^r {\bf d}_n $ and  ${\bf f}_n=\left(f, (\Phi_n^r)^{\rm T}\right)$.
 Then the matrix equation for (\ref{approximationsolution})  is
\begin{eqnarray}\label{noprecondiitioning}
{\bf A} {\bf d}_n={\bf f}_n.
\end{eqnarray}
Meanwhile, by (\ref{waveletbasissmatrix1}) and (\ref{waveletbasissmatrix1}), the matrix equation for the basis functions $\widetilde{\Psi}^r_n$
  actually is
\begin{eqnarray}\label{preconditioningsystem}
{\left(\widetilde{{\bf D}}^r_n ({\bf M}^r)^T{\bf A}{\bf M}^r \widetilde{{\bf D}}^r_n \right)}
{\left(({\bf M}^r\widetilde{{\bf D}}^r_n)^{-1}{\bf d}_n\right)}={\widetilde{{\bf D}}^r_n ({\bf M}^r)^T{\bf f}_n}.
\end{eqnarray}
The system (\ref{preconditioningsystem}) can be regarded as the preconditioned form of the system (\ref{noprecondiitioning}).
 Since the condition number of matrix $\widetilde{{\bf D}}^r_n ({\bf M}^r)^T{\bf A}{\bf M}^r \widetilde{{\bf D}}^r_n $ is uniformly bounded,
 if the conjugate gradient method (CG) is used, the iteration number will be independent of the size of ${\bf d}_n$ \cite{Chan:07}.
 The  CG method for (\ref{preconditioningsystem}) can be performed like the programs provided in \cite{Chan:07},  where in each iteration,
  the matrix vector products like ${\bf M}^r {\bf e}, ({\bf M}^r)^{\rm T} {\bf e}$, ${\bf D}^r_n{\bf e}$, and ${\bf A}{\bf e}$ are needed,
   but in fact, they can be performed effectively with the total cost $\mathcal{O}(N\log N)\,(N=2^n-r+1)$. More specifically,
\begin{itemize}
\item ${\bf D}^r_n$ is a diagonal matrix, which can be generated with  the cost $\mathcal{O}(\log_2(N))$, and stored with the  cost $\mathcal{O}(N)$.
\item ${\bf M}^r$ and $ ({\bf M}^r)^{\rm T}$  are usually called the fast wavelet transform (FWT) matrices.
They do not need to be pre-stored or assembled, and  ${\bf M}^r {\bf e} $ and $({\bf M}^r)^{\rm T} {\bf e}$ can be implemented
following a process like  \cite[pp. 431]{Cohen:00}, with the cost $\mathcal{O}(N)$.
\item ${\bf A}$ is a Toeplitz matrix, so the storage cost is $\mathcal{O}(N)$, and  by the FFT, the computation cost
for ${\bf A}{\bf e}$ is $\mathcal{O}(N\log N)$ \cite[pp. 11-12]{Chan:07} and \cite{Wang:10}.
\end{itemize}

\section{Weak solutions for problems with generalized Dirichlet type boundary condition}\label{sec6}
Like the existing literatures on variational numerical methods for non-local diffusion problems \cite{Acosta:171,Du:12,DElia:13,Ervin:16,Xu:14,Tian:16},  we  have discussed numerical methods for (\ref{Laplacedirichelt}) with the homogeneous boundary condition in the previous sections. In this section, we consider the problem with generalized Dirichlet type boundary condition, i.e.,
\begin{eqnarray}\label{Laplacedirichelt22}
\left\{\begin{array}{ll}
-(\Delta +\lambda)^ {\beta/2}p(x)=f(x),&~~ x\in \Omega,\\
p(x)=g(x),&~~x\in \mathbb{R}\backslash\Omega.
\end{array}\right.
\end{eqnarray}

Introducing a function $\eta(x)$ defined in $\mathbb{R}$ such that $\eta(x)=g(x) $ in $\mathbb{R}\backslash \Omega$, the weak solution (\ref{Laplacedirichelt22}) can be defined as: find $p=u+\eta$ such that $u\in {\widetilde{H}_0^{\beta/2}(\Omega)}$ and
\begin{eqnarray}\label{weakdirichletproblem}
B(u,v)=\left\langle f,v\right\rangle-B(\eta,v) ~~~\forall v\in {\widetilde{H}_0^{\beta/2}(\Omega)}.
\end{eqnarray}
\begin{theorem}\label{existenceedd}
Assume that $f\in H^{-\beta/2}(\Omega)$ and there exists a function $\eta(x)$ satisfying $\int_\Omega\int_{\mathbb{R}}\frac{\left(\eta(x)-\eta(y)\right)^2}{|x-y|^{1+\beta}}dxdy<\infty$.
Then (\ref{Laplacedirichelt22}) has an unique weak solution.
\end{theorem}
\begin{proof}
According to the proof of Theorem \ref{Laxgramlemma1}, it remains to show that $B(\eta,\cdot)$ is a bounded linear functional on ${\widetilde{H}_0^{\beta/2}(\Omega)}$. In fact, for any $v\in {\widetilde{H}_0^{\beta/2}(\Omega)}$, it holds that
\begin{eqnarray}
&&\left|B(\eta,v)\right|=\frac{1}{2}\left|\int_\mathbb{R}\int_{\mathbb{R}}\frac{\left(\eta(x)-\eta(y)\right)\left(v(x)-v(y)\right)}{e^{\lambda|x-y|}|x-y|^{1+\beta}}dydx\right|\nonumber\\
&&~~~~~~~~~~~\le \frac{1}{2}\bigg|\int_\mathbb{R}\int_{\Omega}\frac{\left(\eta(x)-\eta(y)\right)\left(v(x)-v(y)\right)}{e^{\lambda|x-y|}|x-y|^{1+\beta}}dydx\nonumber\\
&&~~~~~~~~~~~~+\int_\Omega\int_{\mathbb{R}\backslash \Omega}\frac{\left(\eta(x)-\eta(y)\right)\left(v(x)-v(y)\right)}{e^{\lambda|x-y|}|x-y|^{1+\beta}}dydx\bigg|;
\end{eqnarray}
using the Cauchy-Schwarz inequality yields
\begin{eqnarray}
\left|B(\eta,v)\right|\le \left(\int_{\Omega}\int_{\mathbb{R}}\frac{\left|\eta(x)-\eta(y)\right|^2}{|x-y|^{1+\beta}}dydx\right)^{\frac{1}{2}} \,\left\|v\right\|_{H^{\beta/2}(\mathbb{R})}.
\end{eqnarray}
Thus $(\ref{weakdirichletproblem})$ has an unique solution $u(x)$.

Further, let $\eta, \tilde{\eta}$ be two functions satisfying $\eta=\tilde{\eta}=g$ in $\mathbb{R}\backslash \Omega$, and $p$ and $\tilde{p}$ are the corresponding weak solutions. Then
\begin{eqnarray}\label{ddeeeddff}
B(p-\tilde{p},v)=0~~~\forall v\in {\widetilde{H}_0^{\beta/2}(\Omega)}.
\end{eqnarray}
Choosing $v=p-\tilde{p}$ in (\ref{ddeeeddff}) yields that $p=\tilde{p}$, which means that the weak solution   actually depends only the values  of $g$ in $\mathbb{R}\backslash\Omega$.  Therefore, (\ref{Laplacedirichelt22}) has a unique weak solution $p(x)=u(x)+\eta(x)$.
\end{proof}

For the second order elliptic problem,  since  $\eta(x)=p(0)(1-x)+p(1)x$ satisfies $\eta(0)=p(0), \eta(1)=p(1)$, and $p(x)-\eta(x)\in H_0^1(\Omega)$,  one can easily translate the problem with the general Dirichlet boundary condition  to the problem with zero boundary  (the existence of $\eta(x)$ can also be ensured by the trace theorem). However, for the nonlocal problems with nonlocal boundary conditions, to the best our knowledge, there are no general methods to find  the suitable $\eta(x)$ and no general theory to ensure the existence of $\eta(x)$. Here, we point out that if $g(x)\in L^\infty(\mathbb{R}\backslash\Omega)$, one can take $\eta(x)$ by the following ways to ensure $\int_\Omega\int_{\mathbb{R}}\frac{\left(\eta(x)-\eta(y)\right)^2}{|x-y|^{1+\beta}}dxdy<\infty$:
\begin{enumerate}
  \item If $0<\beta< 1$,  one only needs to extend $g(x)$ such that $\eta(0)=g(0),\eta(1)=g(1),\eta(x)\in H^{\beta/2}(\Omega)$, and $\left\|\eta\right\|_{L^\infty(\Omega)}<\infty$.  In particular, the  function $S_1(x)=g(0)(1-x)+g(1)x$ can be used as $\eta(x)$ for $x\in \Omega$.
  \item If  there exist $a_1<0, \, b_1>1$ such  that  $g(x)$ is one-times continuously differentiable on $[a_1,0]$ and $[1,b_1]$, one only needs to extend $g(x)$ such that $\eta(x)$ is one-times continuously differentiable on $[0,1]$. In particular, the  spline polynomial $S_3(x)$ satisfying
      $S_3(0)=g(0), S_3^{\prime}(0)=g^\prime(0), S_3(1)=g(1), S_3^\prime(1)=g^\prime (1)$ can be used as the $\eta(x)$ for $x\in\Omega$.
\end{enumerate}
In fact, for case 1:
\begin{eqnarray}
&&\int_\Omega\int_{\mathbb{R}}\frac{\left(\eta(x)-\eta(y)\right)^2}{|x-y|^{1+\beta}}dxdy\le \int_\Omega\int_\Omega\frac{\left(\eta(x)-\eta(y)\right)^2}{|x-y|^{1+\beta}}dydx\nonumber\\
&&~~+2\left(\left\|g(x)\right\|_{L^\infty(\mathbb{R}\backslash \Omega)}^2+\left\|\eta\right\|_{L^\infty(\Omega)}^2\right)\int_\Omega\int_{\mathbb{R}\backslash\Omega}\frac{1}{|x-y|^{1+\beta}}dydx
<\infty.
\end{eqnarray}
For case 2:
\begin{eqnarray*}
&&\int_\Omega\int_{\mathbb{R}}\frac{\left(\eta(x)-\eta(y)\right)^2}{|x-y|^{1+\beta}}dxdy\le \int_\Omega\int_{a_1}^{b_1}\frac{\left(\eta(x)-\eta(y)\right)^2}{|x-y|^{1+\beta}}dydx\nonumber\\
&&~~+2\left(\left\|g(x)\right\|_{L^\infty(\mathbb{R}\backslash \Omega)}^2+\left\|\eta\right\|_{L^\infty(\Omega)}^2\right)\int_\Omega\int_{\mathbb{R}\backslash(a_1,b_1)}\frac{1}{|x-y|^{1+\beta}}dydx;
\end{eqnarray*}
and by the mean value theorem
\begin{eqnarray*}
\int_\Omega\int_{a_1}^{b_1}\frac{\left(\eta(x)-\eta(y)\right)^2}{|x-y|^{1+\beta}}dydx\le \left\|\eta^\prime\right\|^2_{L^\infty(a_1,b_1)}\int_\Omega\int_{a_1}^{b_1}\frac{1}{|x-y|^{\beta-1}}dydx<\infty.
\end{eqnarray*}
Thus
\begin{eqnarray}
\int_\Omega\int_{\mathbb{R}}\frac{\left(\eta(x)-\eta(y)\right)^2}{|x-y|^{1+\beta}}dxdy<\infty.
\end{eqnarray}

\begin{remark}
 In particular, if $f(x)=0$ and $g(x)=1$ in (\ref{Laplacedirichelt22}), then   $S_1(x)$ and $S_3(x)$ will be $1$.
Thus, one can choose $\eta(x)=1$ for $x\in \mathbb{R}$, and the weak formulation  (\ref{weakdirichletproblem})
reduces to $B(u,v)=0$, which admits an unique solution $u(x)=0$. Therefore, (\ref{Laplacedirichelt22}) has  an unique solution $p=\eta(x)+u(x)=1$.
\end{remark}
 \begin{theorem}\label{convergencedd2}
 Let $u=p-\eta$ with $u\in {\widetilde{H}_0^{\beta/2}(\Omega)}\cap H^\mu(\Omega) \,(\mu\ge \beta/2)$ be the
exact solution of  (\ref{weakdirichletproblem}) and $u_n=p_n-\eta\in V_n$ be the Galerkin  approximation solution. Then
 \begin{eqnarray}
 \left\|p-p_n\right\|_{H^{\beta/2}(\Omega)}\lesssim 2^{-n\left(\min\{\mu, r\}-\beta/2\right)}\left\|p-\eta\right\|_{H^\mu(\Omega)},
 \end{eqnarray}
 where $\beta\in (0,1)$ if $V_n$ is generated from $M_1(x)$, and  $\beta\in (0,2)$ if $V_n$ is generated from $M_2(x)$.
 \end{theorem}

\section{Numerical experiments}\label{sec7}
In this part, we set $\Omega=(0,1)$. The data under `$H^{\beta/2}$-Err' and `$L^2$-Err' are the errors in the norms $\left\|\cdot\right\|_{H^{\beta/2}(\mathbb{R})}$
and $\left\|\cdot\right\|_{L^2(\mathbb{R})}$, respectively. If the true solution is unknown, the `$H^{\beta/2}$-Err' and `$L^2$-Err' are, respectively, replaced by `$\widehat{H}^{\beta/2}$-Err' and `$\widehat{L}^2$-Err', where the errors at level $n$ are defined by
\begin{eqnarray}
\left\|p_{n+1}(x)-p_{n}(x)\right\|_{H^{\beta/2}(\mathbb{R})} ~~{\rm and}~\left\|p_{n+1}(x)-p_{n}(x)\right\|_{L^2(\mathbb{R})},
\end{eqnarray}
respectively, being similar to \cite[Example 5.2]{Deng:16}. We will examine if the computed convergence rates  reflect their counterparts in the $\left\|\cdot\right\|_{H^{\beta/2}(\mathbb{R})}$ and $\left\|\cdot\right\|_{L^2(\mathbb{R})}$ norms, respectively; the convergence rates (i.e., the data under `rate') at level $n$ are calculated by
\begin{eqnarray}
{\rm rate}=\log_2\left(\frac{{\rm the ~error~ with~ solution ~approximated ~in~ }V_{n-1}   }{{\rm the ~error~  with~ solution~ approximated ~in~ }V_{n} }\right).
\end{eqnarray}
\begin{example}\label{example1}
Consider model (\ref{Laplacedirichelt}) with the right-hand side source term $f(x)$ being derived from the exact solution $u(x)=x^2(1-x)$ for $x \in \Omega$.
\end{example}

If $\lambda=0$, the right-hand term $f(x)$  can be explicitly given as
\begin{eqnarray}
f(x)=\frac{1}{\pi}\left(3x-1/2+\left(3x^2- 2x\right)\log\left(\frac{1 - x}{x}\right)\right)
\end{eqnarray}
for $\beta=1$, and
\begin{eqnarray}
&&f(x)=-\frac{c_{\beta}\Gamma(-\beta)}{\Gamma(4-\beta)}\Big(2(3-\beta)x^{2-\beta}-6x^{3-\beta}+6(1-x)^{3-\beta}\nonumber\\
&&~~~~~~~-4(3-\beta)(1-x)^{2-\beta}+(3-\beta)(2-\beta) (1-x)^{1-\beta}\Big)
\end{eqnarray}
 for $\beta\in (0,1)\cup(1,2)$.
If $\lambda\not=0$,  the term $f(x)$ is obtained numerically.  For different $\lambda$ and $\beta$,
the numerical results are listed in Table \ref {table:1-1}, where  in the case $r=2$, the  $\left\|\cdot\right\|_{H^{\beta/2}(\mathbb{R})}$ errors
 for $\lambda=0$ and $\lambda=3$ are almost the same,  and both the $\left\|\cdot\right\|_{H^{\beta/2}(\mathbb{R})}$ convergence rates of $r=1$ and $r=2$
 indeed confirm the theoretical predictions in Theorem  \ref{errorestimate}.

\begin{table}[h t b p]\fontsize{6.0pt}{10pt}\selectfont
\begin{center}
 \caption{Numerical results  for Example \ref{example1} with $r=1$ and $r=2$.}
\begin{tabular}{cc|cc|cc|cc|cc}
  \hline
  $(r,\beta)$  & $n$   &\multicolumn{4}{c|}{$\lambda=0$ }    &\multicolumn{4}{|c}{$\lambda=3$  }\\
             &        & $H^{\beta/2}$-Err      & Rate                  &$L^2$-Err  & rate            &$H^{\beta/2}$-Err  & Rate            &$L^2$-Err & Rate    \\
   \hline

                  &$10$     &1.1942e-03   &  --  & 1.0296e-04       & --    &1.1949e-03     &--    & 1.0382e-04     &-- \\
    $(1,0.3)$     &$11$     &6.6222e-04   &  0.85   & 5.1475e-05      & 1.00   &6.6234e-04     &  0.85     & 5.1648e-05     &1.01\\
                   &$12$     &3.6722e-04   &  0.85   &2.5736e-05     &  1.00   &3.6724e-04     & 0.85    &2.5771e-05  &1.00 \\
                      \\[-5pt]

                        &$10$     &1.1902e-02   &--   &  2.1674e-04      &  --    &1.2002e-02     & --    & 7.0126e-04    &-- \\
    $(1,0.8)$         &$11$     &7.8613e-03    &  0.60    &9.8403e-05      &    1.14      &7.8908e-03    &  0.60     &3.0980e-04     &1.18 \\
                       &$12$     &5.1894e-03   &  0.60   &   4.4872e-05     &   1.13     &5.1980e-03     & 0.60    & 1.3638e-04     &1.18 \\

                       \hline
                     &$9$     &4.9203e-06  &  --   &  2.8654e-07    &  --    &4.9204e-06     & --    & 2.8687e-07     &-- \\
    $(2,0.5)$        &$10$     & 1.4591e-06    &  1.75   & 7.1360e-08     &   2.01   &1.4591e-06     & 1.75   & 7.1392e-08    &2.00 \\
                       &$11$     &4.3320e-07   &   1.75   & 1.7805e-08   &  2.00    &4.3320e-07     & 1.75    & 1.7808e-08     &2.00 \\
                       \\[-5pt]

                     &$9$     &3.0192e-05    &  --        &  2.8876e-07    &  --    &3.0193e-05     & --   & 2.9096e-07     &-- \\
    $(2,1.0)$        &$10$     & 1.0662e-05    & 1.50     & 7.1677e-08     &  2.01  &1.0662e-05     & 1.50   & 7.1959e-08     &2.01 \\
                   &$11$     &3.7641e-06   &  1.50  &   1.7850e-08    &  2.00   & 3.7641e-06    & 1.50    & 1.7886e-08    &2.00 \\
                       \\[-5pt]

                     &$9$     &1.1791e-03  &  --    & 4.3811e-07     &  --    &1.1791e-03     & --   & 4.6380e-07     &-- \\
    $(2,1.8)$      &$10$     &5.4999e-04   &  1.10   & 1.0533e-07     &  2.06   &5.4999e-04     &1.10  & 1.1117e-07     &2.06 \\
                   &$11$     &2.5655e-04   &  1.10   & 2.5385e-08     &  2.05    &2.5655e-04    & 1.10    & 2.6708e-08     &2.06 \\

\hline
\end{tabular} \label{table:1-1}
\end{center}
\end{table}

 The condition numbers of systems (\ref{noprecondiitioning}) and (\ref{preconditioningsystem}) and the corresponding iterations
 of the  conjugate gradient (CG) methods  (run in  MATLAB $7.0$) are presented in Table \ref{table:1-3}, where `Gauss' denotes the Gaussian elimination method,
 and  the `CG' and `PCG' denote the CG iterations for solving systems (\ref{noprecondiitioning}) and (\ref{preconditioningsystem}), respectively.
 The stopping criterion for the iteration methods is
 \[ \frac{\|R(k)\|_{l_2}}{\|R(0)\|_{l_2}}\le 1e-9,\]
 with $R(k)$ being the residual vector of linear systems after $k$ iterations.
The comparisons for the three methods are made almost with the same $L_2$ approximation errors, not listed in the table.
 One can see that without preconditioning, the condition number (see the data under `Cond') of
the stiffness matrix behaves like $\mathcal{O}(2^{n\beta})$,  and the iteration numbers (see the data under `iter') increase with $n$, especially when $\beta$ is big. After preconditioning, uniformly bounded condition numbers  are obtained,
 and the iteration numbers of  the CG method are essentially independent of $n$. We also display the eigenvalue distributions
of the stiffness matrices for $(\beta,r)=(0.3,1),\, (\beta,r)=(0.8,1),\, (\beta,r)=(1,2)$,
and $(\beta,r)=(1.8, 2)$ in Figure  \ref{figure:1-1}; they show the preconditioning benefits of  a more concentrated eigenvalue distribution.

 \begin{table}[h t b p]\fontsize{6.0pt}{10pt}\selectfont
\begin{center}
 \caption{The condition numbers and iteration performances of the conjugate gradient method for Example \ref{example1}  with $ \lambda=3$.}
\begin{tabular}{cc|cccc|ccc|c}
  \hline
   $(r,\beta)$  & $n$    &\multicolumn{4}{c}{CG }    &\multicolumn{3}{c}{PCG } & Gauss \\
             &         & $\# $Cond      & $\#$rate   &$\#$ iter  & CPU(s)    &  $\# $Cond        &$\#$ iter  & CPU(s) &CPU(s)            \\

          \hline
                  &$11$     & 1.5869e+02  &0.33   &  60   &  0.0290    &   9.7580   &23      & 0.0362     & 0.1457 \\
     $(1,0.3)$    &$12$     & 1.9934e+02  &0.33   &  67    & 0.0688    & 9.8138    &23      & 0.0558    &1.0028 \\
                 &$13$        &2.4939e+02  &0.32   &  75   & 0.1224     &  9.8564    &23       & 0.0734     &7.0545 \\
     \\[-5pt]
                 &$11$      & 5.7500e+02  &0.51   &  107  & 0.1120     &  15.896   &30      & 0.0513    &0.1532 \\
    $(1,0.5)$    &$12$     &  8.1873e+02 &0.50  &  128    & 0.1095    &  16.156   &31      &  0.0691     & 1.0320 \\
                 &$13$     &1.1634e+03  &0.50   &  152    & 0.2819    &  16.379    &32      & 0.0997    &7.0181 \\

                      \\[-5pt]

                 &$11$      & 6.5001e+03  &--   &  318   & 0.1535     & 53.005   &55    &0.0945    &0.1452 \\
    $(1,0.8)$    &$12$     & 1.1355e+04 & 0.80   &  422   & 0.3844    &  55.469   &58        & 0.1417  &  0.9792 \\
                 &$13$     & 1.9803e+04 & 0.80   &  559  & 0.8489    &   55.468   &60      & 0.2120    &7.1678 \\

 \hline
                 &$10$      &1.3886e+02 &--   &  51  & 0.0334    &  4.8126    &22       & 0.0165   &0.0207\\
    $(2,0.5)$    &$11$     &  2.0043e+02 &0.52   &  62    & 0.0688    & 4.8401   &22     &  0.0409     &0.1420 \\
                 &$12$     & 2.8753e+02  &0.52   &  75    & 0.0817    & 4.8585    &22      & 0.0402    &0.9726\\
  \\[-5pt]

                 &$10$     & 1.5838e+03 &--   &  141     & 0.0576   &  6.1738      &27      &  0.0286      &0.0327 \\
    $(2,1.0)$     &$11$     & 3.1773e+03 &1.00   & 200    & 0.2008    & 6.2427    &28      &0.0520       &0.1479 \\
                  &$12$     &  6.3654e+03 &1.00   &  285   & 0.2522    &  6.2961    &28      & 0.0575   &1.0236 \\
 \\[-5pt]
                 &$10$      & 2.5703e+04  &--   &  421   &0.1754     & 8.7180     &34      & 0.0365      &0.0300 \\
    $(2,1.5)$    &$11$     & 7.2749e+04 &1.50  &   710     & 0.7271    & 8.7749    &35        & 0.0633     &0.1573 \\
                 &$12$     & 2.0584e+05  &1.50  & 1198    & 1.1731    &  8.8212    &35        & 0.0740    &1.0176 \\
                      \\[-5pt]

                 &$10$      & 1.3923e+05  &--   & 767  & 0.3242     &  11.993   &40       & 0.0422    &0.0301 \\
    $(2,1.8)$    &$11$     & 4.8494e+05 &1.80   &  1432    &1.4589     & 12.138  &41      & 0.0727     &0.1522 \\
                 &$12$     & 1.6889e+06  &1.80   & 2674    &2.6306     & 12.256    &42     & 0.0836     &0.9777
 \\
\hline
\end{tabular}\label{table:1-3}
\end{center}
\end{table}

 \begin{figure}[h t b p]
\begin{center}
\includegraphics[width=1.1in,height=1.1in,angle=0]{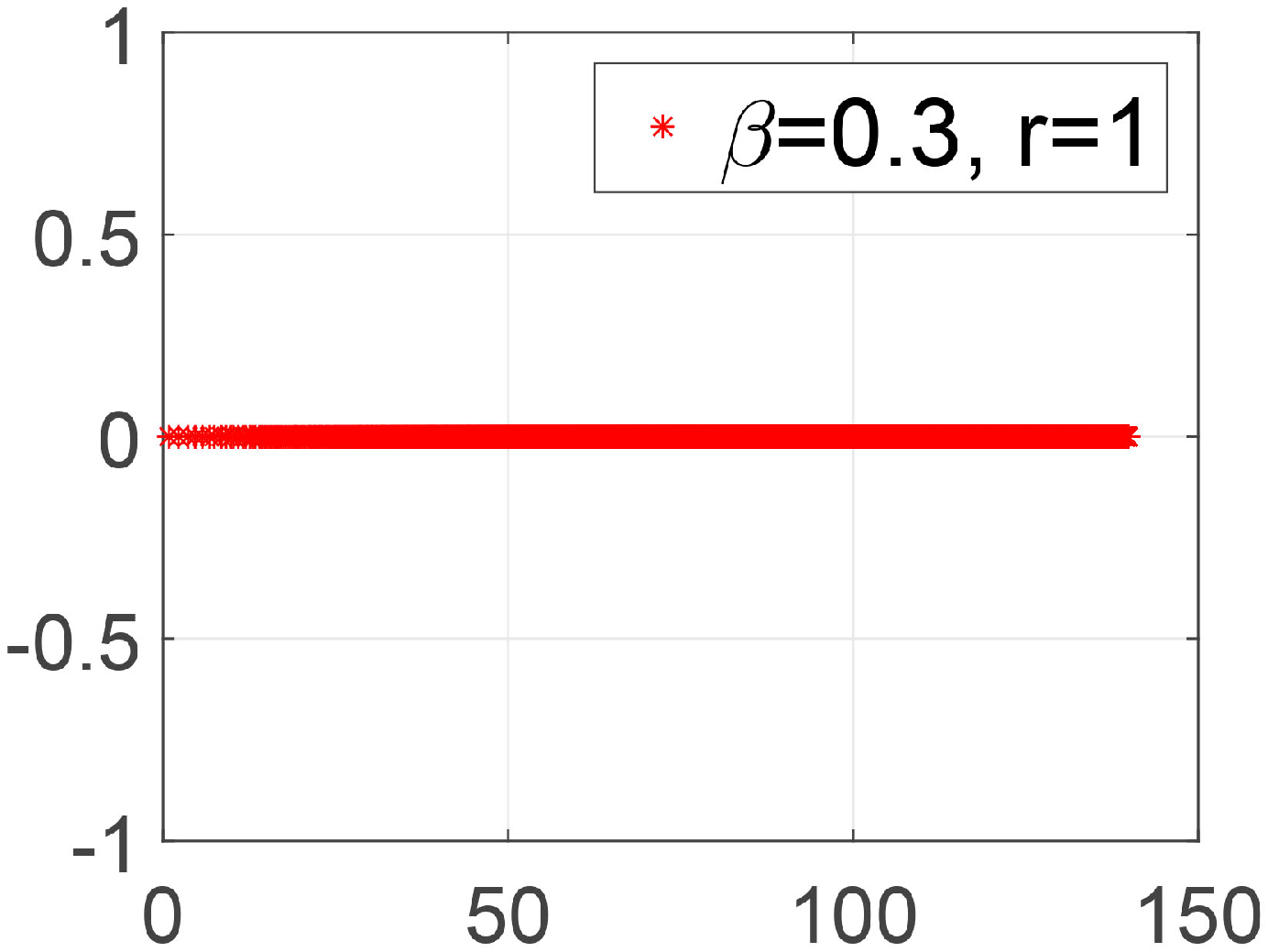}
\includegraphics[width=1.1in,height=1.1in,angle=0]{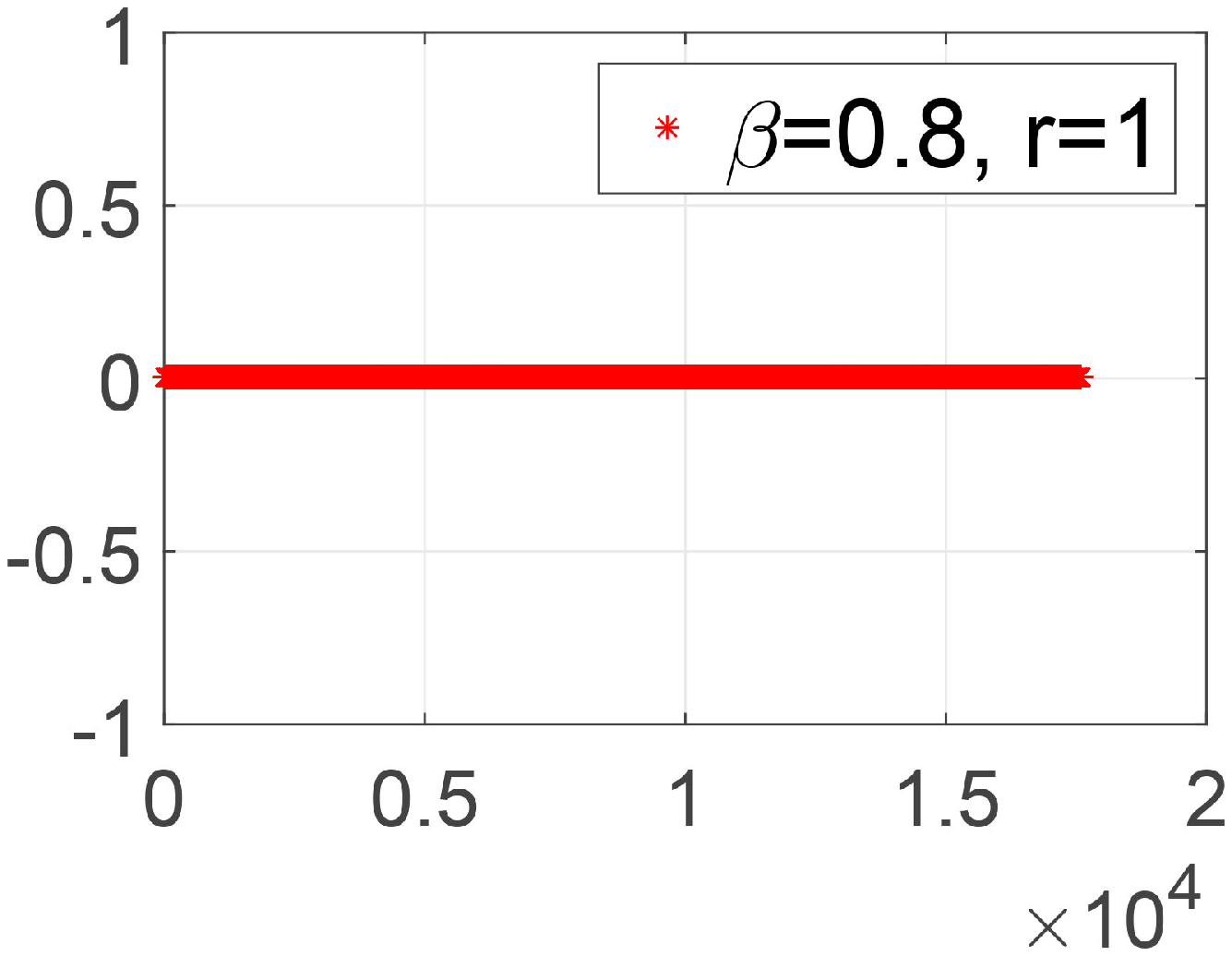}
\includegraphics[width=1.1in,height=1.1in,angle=0]{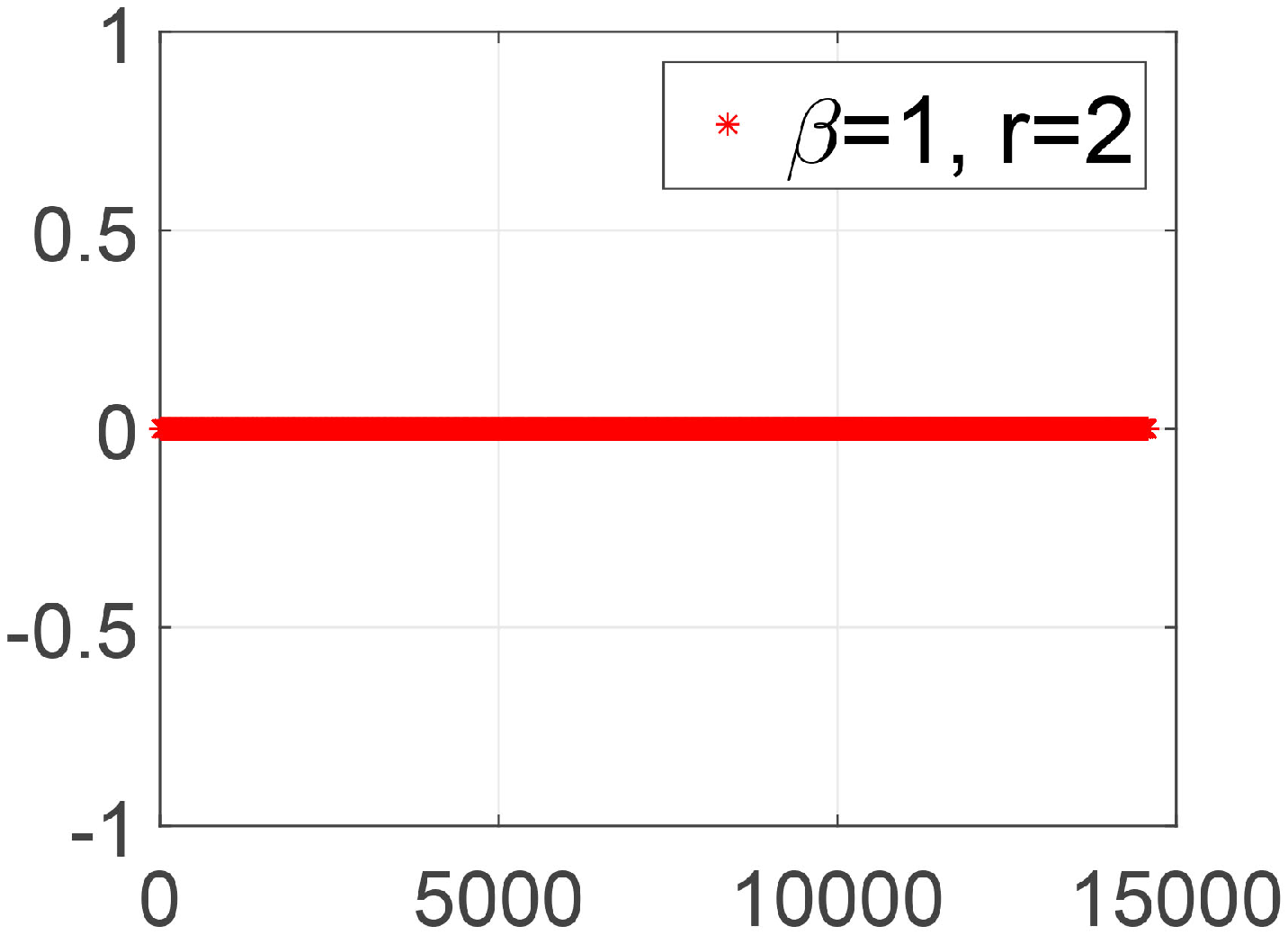}
\includegraphics[width=1.1in,height=1.1in,angle=0]{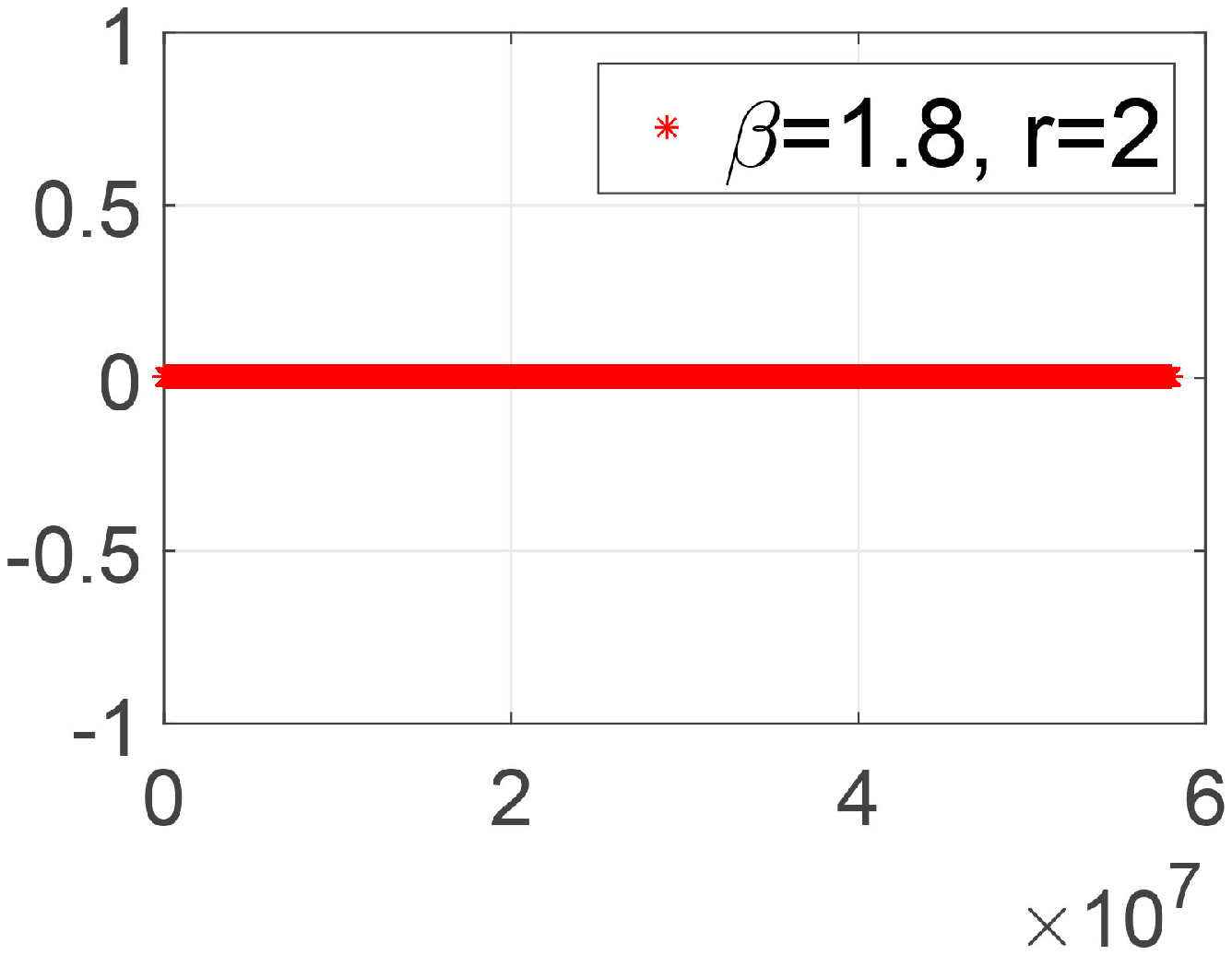}\\
\includegraphics[width=1.1in,height=1.1in,angle=0]{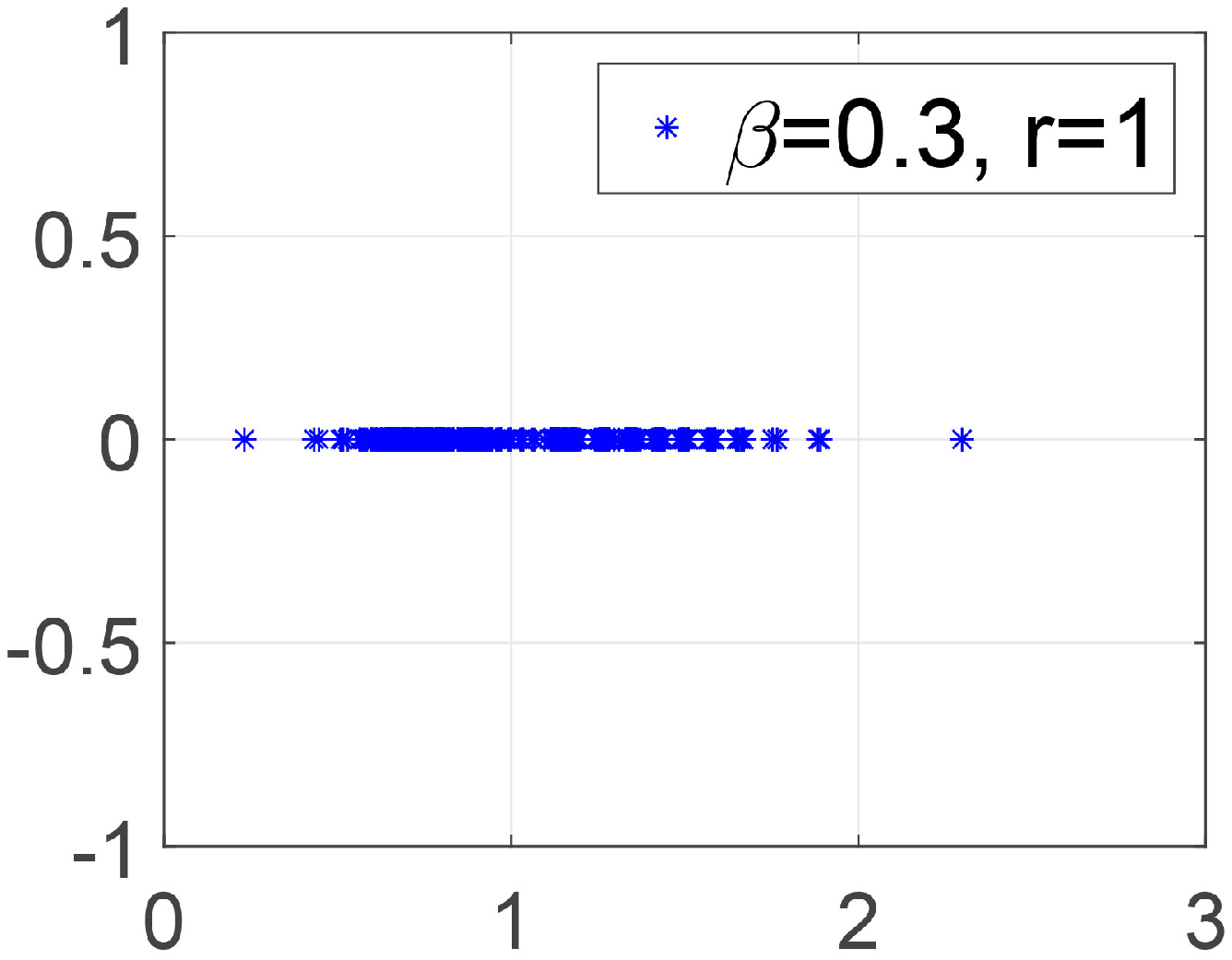}
\includegraphics[width=1.1in,height=1.1in,angle=0]{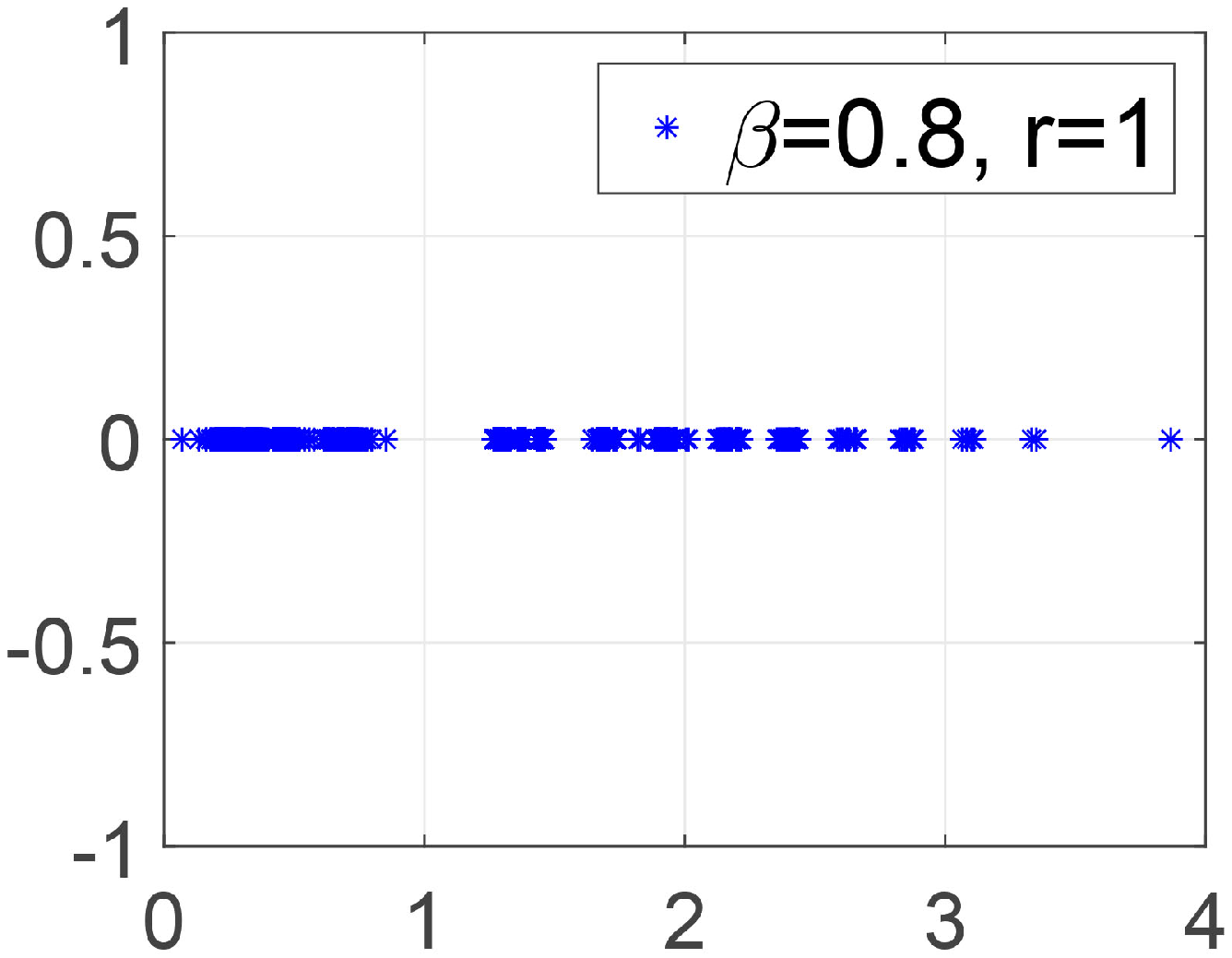}
\includegraphics[width=1.1in,height=1.1in,angle=0]{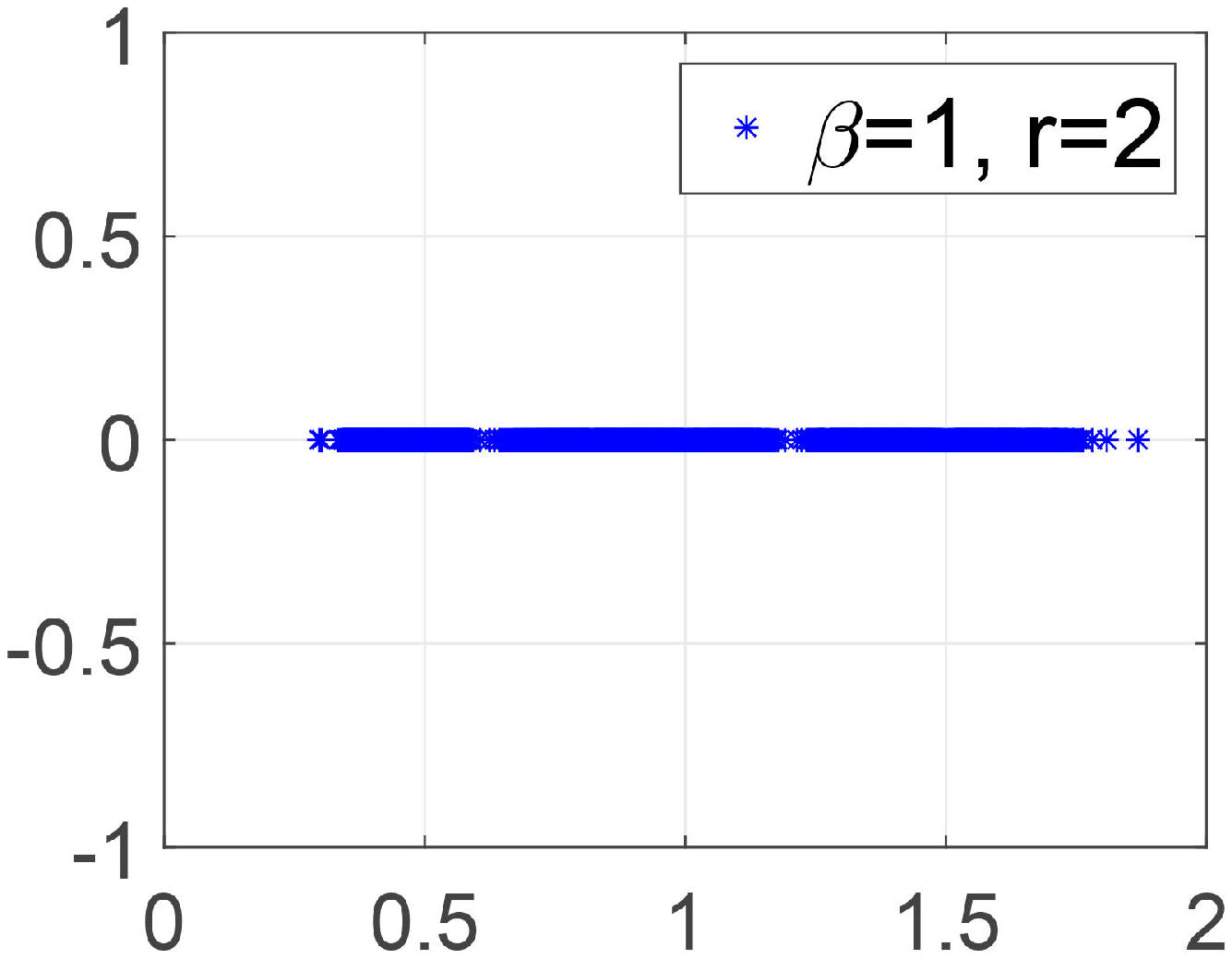}
\includegraphics[width=1.1in,height=1.1in,angle=0]{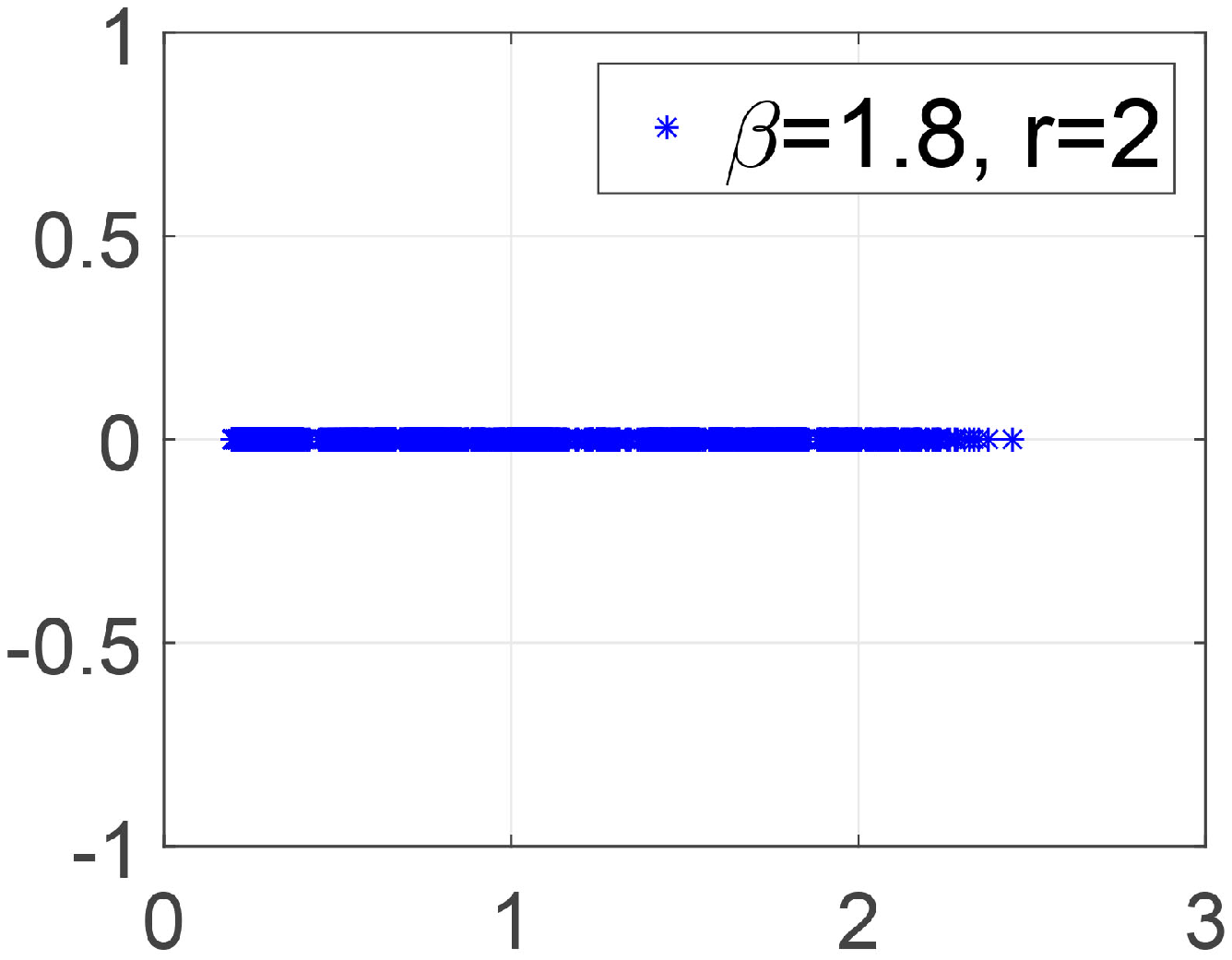}
\caption{Eigenvalue distribution of the systems (\ref{noprecondiitioning}) and (\ref{preconditioningsystem})
(for $n=12$) with the single scaling basis functions (first line) and  the corresponding multiscale Reisz basis functions
 (second line), respectively. The horizontal and vertical axes are respectively the real and imaginary axis.
 }\label{figure:1-1}
\end{center}
\end{figure}

\begin{example}\label{example2}
We now take $f(x)=1$ in model (\ref{Laplacedirichelt}).
\end{example}

If $\lambda=0$, for $x\in \Omega$, the exact solution is $p(x)=\frac{(x-x^2)^{\beta/2}}{\Gamma(1+\beta)}$ . Although the right-hand side is smooth, $p(x)$ just belongs to $H^{\beta/2+1/2-\epsilon}(\mathbb{R})$ for any $\epsilon>0$. The numerical results are listed in Table \ref{table:2-1},
where  the  predicted  $1/2-\epsilon$ order of convergence in the  $\left\|\cdot\right\|_{H^{\beta/2}(\mathbb{R})}$  norm  by Theorem \ref{errorestimate}
is   obtained. The  $L^2$ convergence orders  $(1+\beta)/2$  for $\beta\in (0,1]$ and $1$ for $\beta\in (1,2)$ confirm the result given in \cite[Proposition 4.3]{Borthagaray:16} for $\lambda=0$. When $\lambda\not=0$,  $p(x)$ cannot be obtained explicitly
so we list the $\widehat{H}^{\beta/2}$ and $\widehat{L}^2$ errors instead, and examine if the convergence rates reflect the convergence rates in the   $\left\|\cdot\right\|_{H^{\beta/2}(\mathbb{R})}$
and $\left\|\cdot\right\|_{L^2(\mathbb{R})}$ norms, respectively.
The numerical results are presented in  Table \ref{table:2-2}, suggesting that the exact solution has a low regularity, but this needs to be confirmed by more in-depth analysis.
\begin{table}[h t b p]\fontsize{7.0pt}{10pt}\selectfont
\begin{center}
 \caption{Numerical results  for Example \ref{example2} with $r=2$ and $\lambda=0$.}
\begin{tabular}{cc|cc|cc|cc|cc}
  \hline
        $\beta$     &        & $H^{\beta/2}$-Err      & Rate                  &$L^2$-Err  & rate            &$\widehat{H}^{\beta/2}$-Err  & Rate            &$\widehat{L}^2$-Err & Rate    \\
   \hline
                     &$8$     &9.0822e-02   & --   & 7.8260e-03     & --    &6.4202e-02     & --    & 5.3795e-03     &-- \\
    $0.5$    &$9$     &6.4156e-02   &   0.50   & 4.6533e-03    &   0.75   &4.5349e-02    &0.50   &3.1985e-03     &0.75 \\
                     &$10$     &4.5340e-02  &   0.50   & 2.7668e-03     &   0.75    &3.2042e-02    & 0.50    &1.9019e-03     &0.75 \\

                       \\[-5pt]

                    &$8$     &4.7148e-02   & --    & 1.1967e-03     & --    &3.3350e-02     & --    & 7.4060e-04     &-- \\
    $1.0$    &$9$     &3.3320e-02    &  0.50  & 6.1260e-04     & 0.97   &2.3565e-02     & 0.50  & 3.7597e-04     &0.98 \\
                  &$10$     &2.3554e-02   &  0.50   & 3.1330e-04     &  0.97   &1.6657e-02     & 0.50    &1.9082e-04     &0.98 \\
                        \\[-5pt]

                    &$8$     &2.2624e-02   &  --     & 1.7078e-04     &  --    &1.6039e-02     & --   & 9.3174e-05   &-- \\
    $1.5$    &$9$     &1.5956e-02  &  0.50      & 8.3518e-05      &1.03    &1.1297e-02     & 0.50   & 4.4561e-05    & 1.06 \\
                       &$10$     &1.1268e-02  &  0.50   & 4.1119e-05    &  1.02    &7.9727e-03     & 0.50    & 2.1564e-05    & 1.05 \\

                    \hline

\end{tabular}\label{table:2-1}
\end{center}
\end{table}

\begin{table}[h t b p]\fontsize{7.0pt}{10pt}\selectfont
\begin{center}
 \caption{Numerical results for Example \ref{example2} with $r=2$ and $\lambda\not=0$.}
\begin{tabular}{cc|cc|cc|cc|cc}
  \hline
   $$  & $n$   &\multicolumn{4}{c|}{$\lambda=1.5$ }    &\multicolumn{4}{|c}{$\lambda=3$  }\\
        $\beta$     &        & $\widehat {H}^{\beta/2}$-Err      & Rate                  &$\widehat{L}^2$-Err  & rate            &$\widehat{H}^{\beta/2}$-Err  & Rate            &$\widehat{L}^2$-Err & Rate    \\
   \hline
                     &$8$     &3.1127e-02   &  --    & 2.7167e-03      &  --   &4.3713e-02     & --    & 3.9690e-03    &-- \\
    $0.5$    &$9$     &2.1930e-02  &    0.50      &1.5970e-03      &  0.77   &3.0757e-02     & 0.51   & 2.3121e-03     &0.78\\
                     &$10$     &1.5467e-02   &    0.50   &  9.4088e-04     & 0.76   &2.1674e-02     & 0.50    &1.3514e-03    &0.77 \\
                        \\[-5pt]

                    &$8$     &1.6877e-02  & --    & 4.2785e-04     & --   &2.0984e-02    & --   & 5.8544e-04     &-- \\
    $1.0$    &$9$     &1.1972e-02   &  0.50        &2.1812e-04     & 0.97   &1.4924e-02     &  0.49   &2.9984e-04     &0.97 \\
                 &$10$     &8.4834e-03  &  0.50       & 1.1086e-04     & 0.98    &1.0593e-02     & 0.49    & 1.5273e-04     &0.97 \\
                        \\[-5pt]

                    &$8$     &5.8889e-03  &  --    & 3.9271e-05       & --    &6.5308e-03     & --      &  4.7041e-05     &-- \\
    $1.5$    &$9$     &4.1724e-03   &  0.50  & 1.9185e-05       &  1.03         &4.6490e-03     &  0.49   & 2.3304e-05    &1.01 \\
                   &$10$     &2.9553e-03   & 0.50   & 9.4274e-06   &  1.02    &3.3025e-03     &  0.49    & 1.1562e-05    &1.01 \\

                    \hline

\end{tabular}\label{table:2-2}
\end{center}
\end{table}

\begin{example}\label{example3}
In this example, model (\ref{Laplacedirichelt22}) is considered in two cases.

\end{example}
For the first case, let $g(x)$ and $f(x)$ in (\ref{Laplacedirichelt22}) be the functions derived from the postulated exact solution $p(x)=e^{-x^2}$. Note that $\mathscr{F}[e^{-x^2}](\xi)=\sqrt{\pi}e^{-\xi^2/4}$. Then $p(x)\in H^\mu(\mathbb{R})$ for any $\mu\ge 0$.
The numerical results for $r=1$ and $\eta(x)=S_1(x)=\left(e^{-1}-1\right)x+1$ are presented in Figure \ref{figure:3.1}, and show $\frac{2-\beta}{2}$th order convergence in $\left\|\cdot\right\|_{H^{\beta/2}(\mathbb{R})}$ norm and  first order convergence in the $\left\|\cdot\right\|_{L^2(\mathbb{R})}$ norm.
\begin{figure}[h t b p]
\centering
\includegraphics[width=0.45\textwidth]{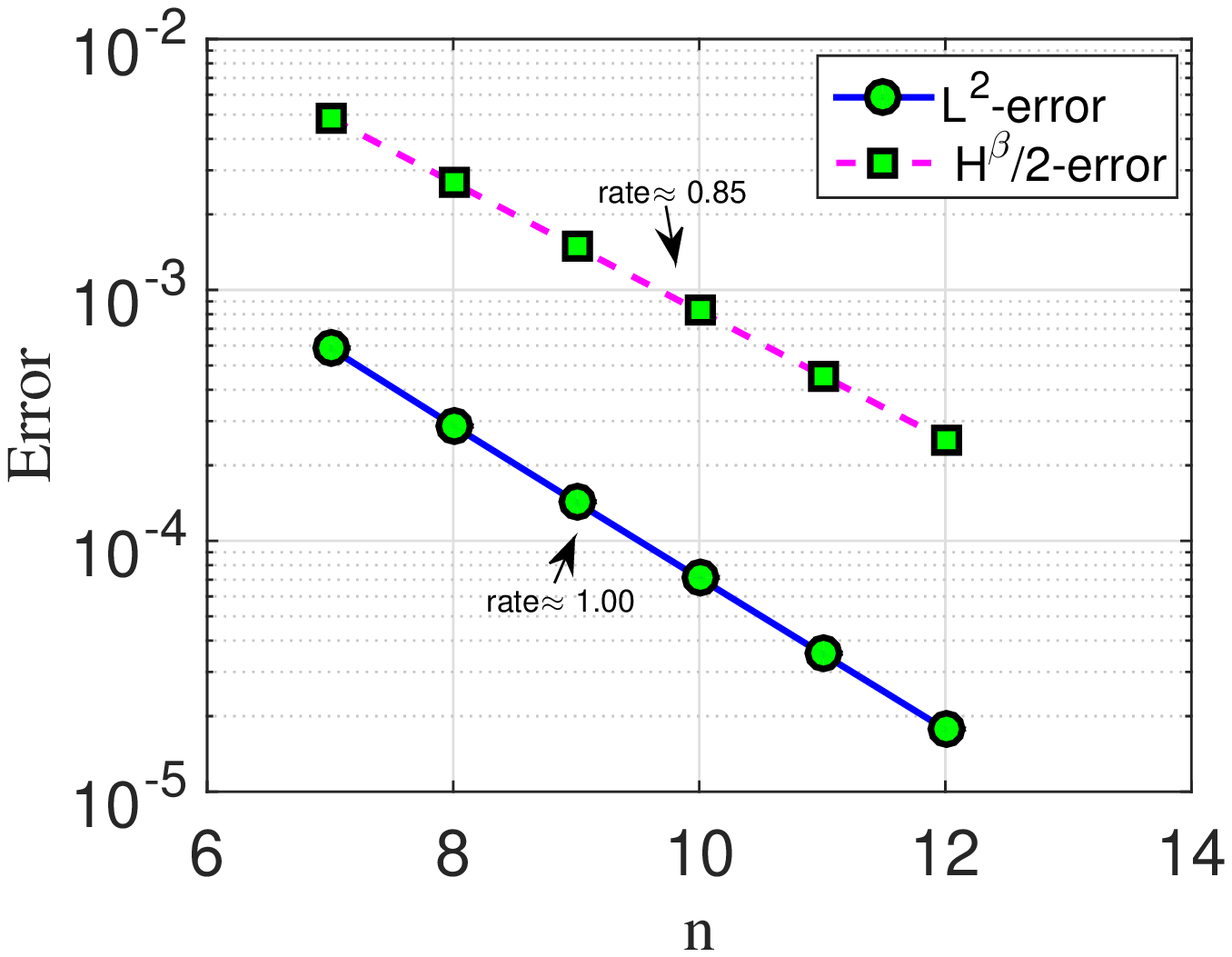}
\includegraphics[width=0.45\textwidth]{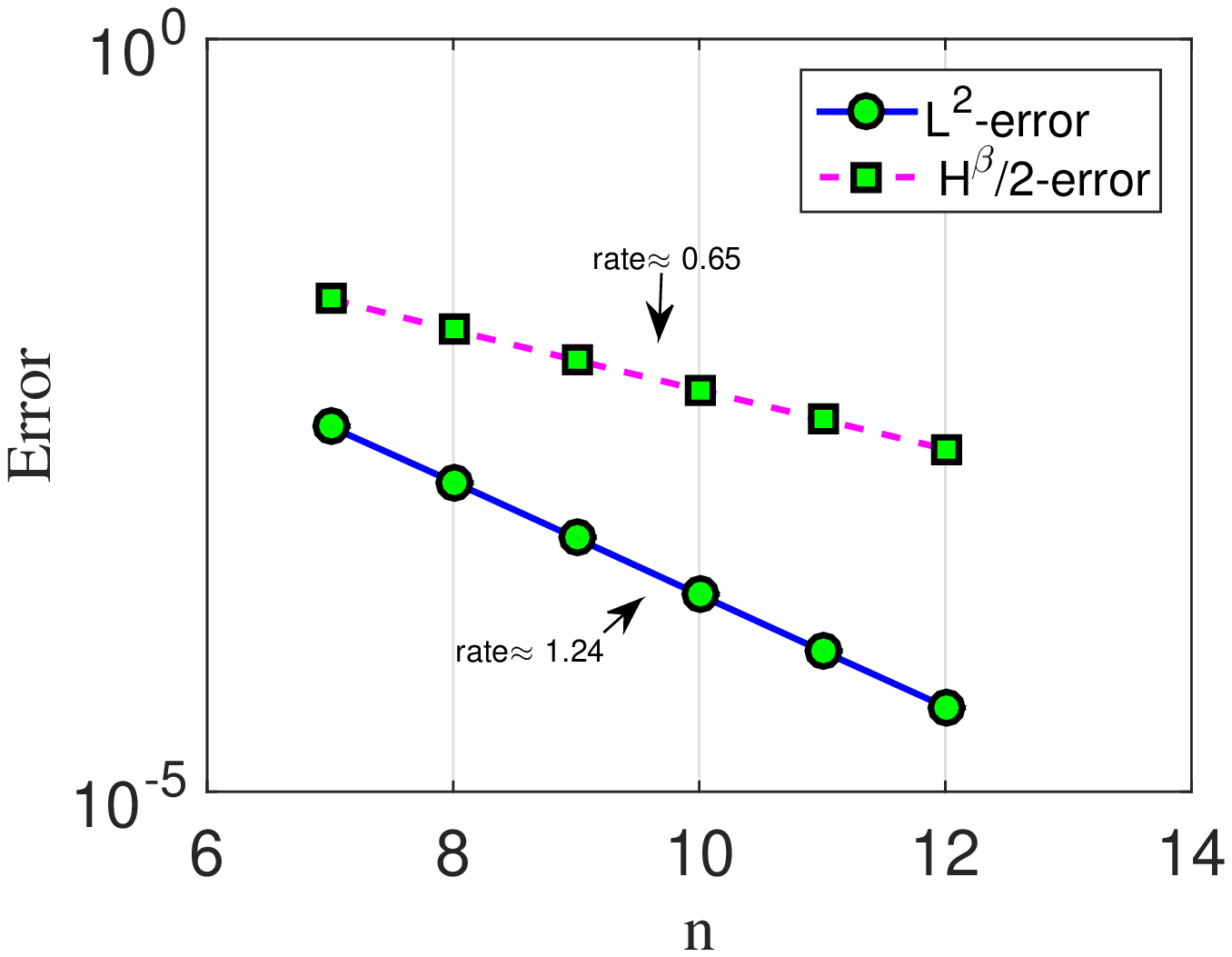}
\caption{Numerical results for  Example \ref{example3} with  $p(x)=e^{-x^{2}}, \,\eta(x)=S_1(x)$,  and $r=1$. The left one is for $\beta=0.3$ and $\lambda=1.5$, and the right one for $\beta=0.7$ and $\lambda=3$.}
\label{figure:3.1}
\end{figure}

For the second case, consider model (\ref{Laplacedirichelt22}) with the generalized  Dirichlet type boundary condition
\begin{eqnarray}\label{geeddgex}
g(x)=\left\{\begin{array}{lll}
2x-2,&x\in [1,\frac{3}{2}],\\[4pt]
-2x,&x\in [-\frac{1}{2},0],\\[4pt]
0,& x\in (-\infty,-\frac{1}{2})\cup(\frac{3}{2},\infty),
\end{array}\right.
\end{eqnarray}
and the source term $f(x)$ being derived from the exact solution
\begin{eqnarray}
p(x)=\left\{\begin{array}{lll}
2x-2,&x\in [1,\frac{3}{2}],\\[4pt]
(x-x^2)^2,&x\in (0,1),\\
-2x,&x\in [-\frac{1}{2},0],\\[4pt]
0,& {\rm else}.
\end{array}\right.
\end{eqnarray}
Obviously, $p(x)$ does not belong to $H^{\mu}(\mathbb{R})$ for $\mu>1/2$ because of its  discontinuity at $x=3/2$ and $x=-1/2$.
We consider two different $\eta(x)$, i.e., the  $\eta(x)=S_2(x)=0$ for $x\in \Omega$, and the $\eta(x)=S_3(x)=2x(x-1)$ for $x\in \Omega$.
Note that both of them satisfy $\int_\Omega\int_{\mathbb{R}}\frac{\left(\eta(x)-\eta(y)\right)^2}{|x-y|^{1+\beta}}dxdy<\infty$, required in Theorem \ref{existenceedd}, but do not belong to $H^{\beta/2}(\mathbb{R})$ (the requirement in \cite[Subsection: 4.1]{Deng:17}) for $\beta>1$, which implies that the condition in Theorem \ref{existenceedd} is weaker than the one of \cite[Subsection: 4.1]{Deng:17}.
The numerical results are presented in Table \ref{table:3-1}, and also confirm the theoretical prediction of Theorem \ref{convergencedd2}.

\begin{table}[h t b p]\fontsize{6.0pt}{10pt}\selectfont
\begin{center}
 \caption{Numerical results  for Example \ref{example3} with $g(x)$ given as (\ref{geeddgex})  and $r=2$.}
\begin{tabular}{cc|cc|cc|cc|cc}
  \hline
  $(\eta,\beta)$  & $n$   &\multicolumn{4}{c|}{$\lambda=0$ }    &\multicolumn{4}{|c}{$\lambda=3$  }\\
             &        & $H^{\beta/2}$-Err   & Rate    &$L^2$-Err  & rate      &$H^{\beta/2}$-Err  & Rate            &$L^2$-Err & Rate    \\
   \hline

                    &$9$     &2.2094e-06   &  --    & 1.2913e-07    &  --    &2.2096e-06    & --   & 1.2974e-07      &-- \\
    $(\eta_1,0.5)$    &$10$   & 6.5387e-07   &  1.76   & 3.2038e-08    &  2.01   &6.5388e-07    & 1.76   &  3.2096e-08     &2.02 \\
                     &$11$     &1.9393e-07  &  1.75   & 7.9785e-09    &  2.00    & 1.9393e-07    & 1.75    & 7.9839e-09     &2.01 \\
                       \\[-5pt]

                    &$9$     &1.0046e-05   &  --    &  5.8314e-07     &  --    & 1.0046e-05     & --     & 5.8329e-07     &--\\
    $(\eta_2,0.5)$    &$10$    &2.9847e-06   &  1.75   & 1.4573e-07     & 2.00   &2.9847e-06     & 1.75   & 1.4574e-07     &2.00 \\
                       &$11$    &8.8697e-07  & 1.75   & 3.6426e-08    & 2.00    &8.8697e-07     & 1.75   & 3.6427e-08     &2.00 \\

    \hline

                    &$9$     &1.3532e-05     &  --    &  1.3167e-07     &  --    & 1.3532e-05     & --   & 1.3623e-07    &-- \\
    $(\eta_1,1.0)$    &$10$    & 4.7735e-06   &  1.50  &  3.2397e-08     &   2.02   &4.7734e-06     & 1.50   & 3.3008e-08    &2.04 \\
                    &$11$    &1.6846e-06      & 1.50  & 8.0282e-09     &  2.01   &1.6846e-06     & 1.50   & 8.1080e-09     &2.02 \\

                           \\[-5pt]

                    &$9$     &6.1750e-05   &  --     & 5.8376e-07     &  --    &6.1750e-05    & --   & 5.8506e-07     &-- \\
    $(\eta_2,1.0)$    &$10$     & 2.1824e-05   &    1.50  &  1.4581e-07     &  2.00   &2.1824e-05     &1.50   &1.4597e-07     &2.00 \\
                       &$11$     &7.6881e-06   &    1.51   &3.6437e-08     &  2.00    & 7.6881e-06     & 1.51   & 3.6457e-08     &2.00 \\

                      \hline

                        &$9$     &1.7807e-04   &  --   &   1.8024e-07     &  --    & 1.7807e-04     &--     & 1.9737e-07     &-- \\
    $(\eta_1,1.6)$     &$10$     & 7.7452e-05   &   1.20   &  4.2847e-08     &  2.07    & 7.7452e-05   &1.20  &  4.6850e-08    &2.07  \\
                       &$11$     & 3.3700e-05   &   1.20   & 1.0224e-08    &  2.06     & 3.3700e-05   &1.20     & 1.1129e-08  &2.07\\

                    \\[-5pt]
                     &$9$     &8.1486e-04   &  --   & 6.1682e-07     &  --    &8.1486e-04     & --   & 6.4349e-07    &-- \\
    $(\eta_2,1.6)$    &$10$     &3.5467e-04   &  1.20   & 1.5186e-07    &  2.02  &3.5467e-04     & 1.20   & 1.5665e-07     &2.04 \\
                       &$11$     &1.5438e-04   & 1.20   & 3.7543e-08     &  2.01    &1.5438e-04    &1.20  &3.8407e-08     &2.03 \\
                    \hline

\end{tabular}\label{table:3-1}
\end{center}
\end{table}

 \section{Conclusions}\label{sec8}

We have presented Riesz basis Galerkin methods for effectively solving the tempered fractional Laplacian equation,
 where the operator is the generator of the tempered $\beta$-stable L\'evy process. The well-posedness of the equation and
convergence of the scheme were theoretically proved. When $\lambda=0$, the model reduces to a fractional Laplacian equation and
the present theoretical framework is still valid.
We also discussed efficient implementations of our methods, including the generation of
stiff matrix and the effectiveness of multiscale preconditioning. We performed several numerical simulations to confirm the theoretical
results and demonstrate the high efficiency of the schemes. The present work is confined to one dimensional problems with basis functions on
uniform meshes. The generalization to higher dimensions and the approximation with locally refined basis functions are very important
topics and will be considered in future work.


\appendix
\section {Proof of $\mathcal{G}(\lambda,\xi,\beta)\ge 0$} \label{appendix}

\begin{proof}
By (\ref{intergraeliqneqq}), it is easy to check that this proof is equivalent to show that	
\begin{eqnarray*}
f(t)=(-1)^{\lfloor \beta\rfloor}\left(\cos\Big(\beta\arctan(t)\Big)-\frac{1}{\left(1+t^2\right)^{\beta/2}}\right)\ge0
\end{eqnarray*}
for $\beta\in (0,1)\cup(1,2)$, and
\begin{eqnarray*}
g(t)=t\arctan(t)-\frac{1}{2}\ln(1+t^2)\ge 0
\end{eqnarray*}
for $\beta=1$,
where $t=\frac{|\xi|}{\lambda}\in[0,\infty)$.
For $\beta=1$, we have $g^{\prime}(t)=\arctan(t)\ge 0$, so $g(t)\ge g(0)=0$. For $\beta\in (0,1)\cup(1,2)$, there exists
\begin{eqnarray*}
f^{\prime}(t)=(-1)^{\lfloor \beta\rfloor}\frac{\beta}{1+t^2}\left(-\sin\Big(\beta\arctan(t)\Big)+\frac{t}{(1+t^2)^{\frac{\beta}{2}}}\right).
\end{eqnarray*}
Thus if $0<\beta<1$, we have
\begin{eqnarray*}
f^{\prime}(t)\ge (-1)^{\lfloor \beta\rfloor}\frac{\beta}{1+t^2}\left(-\sin\Big(\arctan(t)\Big)+\frac{t}{(1+t^2)^{\frac{1}{2}}}\right)=0.
\end{eqnarray*}
Then $f(t)\ge f(0)=0$ for $\beta\in (0,1)$.  If $1<\beta<2$, we have
\begin{eqnarray*}
&&f^{\prime}(t)=(-1)^{\lfloor \beta\rfloor}\Bigg(-\sin\Big((\beta-1)\arctan(t)\Big)\cos\Big(\arctan(t)\Big)\nonumber\\
&&~~-\frac{t}{(1+t^2)^{\frac{1}{2}}}\bigg(\cos\Big((\beta-1)\arctan(t)\Big)-
\frac{1}{\left(1+t^2\right)^{\frac{\beta-1}{2}}}\bigg)\Bigg)\ge0,
\end{eqnarray*}
where the result $f(t)\ge0$ for $\beta\in (0,1)$ has been used to justify the nonnegativity. Then $f(t)\ge f(0)=0$ for $\beta\in(1,2)$.
\end{proof}


\end{document}